\let\ORIlabel\label
\let\ORIrefstepcounter\refstepcounter
\AddToHook{package/hyperref/before}{
   \let\label\ORIlabel 
   \let\refstepcounter\ORIrefstepcounter}

\documentclass[onefignum,onetabnum]{siamart220329}
\allowdisplaybreaks

\usepackage{lipsum}
\usepackage[export]{adjustbox}
\usepackage{amsfonts}
\usepackage{graphicx}
\usepackage{epstopdf}
\usepackage{algorithmic}
\usepackage{cancel}
\usepackage{parskip}
\usepackage{mathtools}
\usepackage{stackengine}
\usepackage{fixmath}
\usepackage{mathrsfs}
\usepackage{dsfont}
\usepackage{amsfonts}
\usepackage{xcolor}
\usepackage{cases}
\usepackage{graphicx} 
\usepackage{enumitem}
\mathtoolsset{centercolon}  
\usepackage{amsmath, amssymb,amsbsy}
\usepackage{bm}

\usepackage{amsmath, amssymb,amsbsy}
\usepackage{mathtools}
\usepackage{stackengine}
\usepackage{fixmath}
\usepackage{mathrsfs}
\usepackage{dsfont}
\usepackage{amsfonts}
\usepackage{xcolor}
\usepackage{cases}
\usepackage{graphicx} 
\usepackage{bm}

\renewcommand{\phi}{\varphi}
\renewcommand{\epsilon}{\varepsilon}
\newcommand{\eps}{\varepsilon}

\newcommand{\cnst}[1]{\mathrm{#1}}

\newcommand{\Id}{\mathbf{I}} 

\DeclareMathOperator*{\argmin}{arg\,min}
\newcommand{\norm}[1]{\left\lVert#1\right\rVert}

\newcommand{\R}{\mathbb{R}}

\newcommand{\N}{\mathbb{N}}
\newcommand{\abs}[1]{\left|#1\right|}
\newcommand{\1}{\mathds{1}}

\newcommand{\Probs}{\mathbb{P}}

\newcommand{\vct}[1]{\mathbold{#1}}

\newcommand{\trace}{\operatorname{tr}}

\newcommand{\diff}{\Phi}

\newcommand{\rmeas}{\mu}

\newcommand{\bpot}{\psi}

\DeclarePairedDelimiterX{\infdivx}[2]{(}{)}{%
  #1\;\delimsize\|\;#2%
}

\newcommand{\divergence}{\operatorname{div}}
\newcommand\D{\mathop{}\cnst{d}}
\newcommand\dd{\cnst{d}} 

\newcommand{\defeq}{\vcentcolon=}
\newcommand{\inprod}[1]{\left\langle #1 \right\rangle}

\newcommand{\helly}{\mathscr{H}} 
\DeclareMathOperator*{\esssup}{ess\,sup}
\DeclareMathOperator*{\essinf}{ess\,inf}
\newcommand{\Leb}{\mathscr{L}}
\newcommand{\Helly}{\helly}
\newcommand{\aederiv}{\hat{\partial}}

\newcommand{\ugen}{v}
\newcommand{\uvel}{u}
\DeclareMathOperator\Ent{\mathsf{Ent}}

\usepackage{tikz}
\usepackage{pgfplots}
\usetikzlibrary{shapes.arrows, patterns, calc}
\usepackage{tikz-3dplot}
\ifpdf
  \DeclareGraphicsExtensions{.eps,.pdf,.png,.jpg}
\else
  \DeclareGraphicsExtensions{.eps}
\fi

\usepgfplotslibrary{groupplots}
\usetikzlibrary{arrows}
\usetikzlibrary{shapes}
\usetikzlibrary{decorations.text}
\usetikzlibrary{quantikz}
\usepackage{siunitx}
\usetikzlibrary{arrows.meta}
\pgfplotsset{ compat=1.18,
    standard/.style={
    scale only axis,
    width=0.5\textwidth,
    enlarge x limits=0.05,
    enlarge y limits=0.05,
    max space between ticks=40,
    every axis/.append style={font=\small},
	every legend/.append style={font=\small},
	every node/.append style={font=\small},	
	}
}

\definecolor{steelblue}{HTML}{A1BDC7}
\definecolor{orange}{HTML}{D98C21}
\definecolor{silver}{HTML}{B0ABA8}
\definecolor{rust}{HTML}{B8420F}
\definecolor{seagreen}{HTML}{2E6B69}
\definecolor{joshua}{HTML}{FBDC7F}
\definecolor{darksky}{HTML}{154c79}

\colorlet{lightsilver}{silver!30!white}
\colorlet{darkorange}{orange!85!black}
\colorlet{darksilver}{silver!85!black}
\colorlet{darksteelblue}{steelblue!85!black}
\colorlet{darkrust}{rust!85!black}
\colorlet{darkseagreen}{seagreen!85!black}

\hypersetup{colorlinks=true,linkcolor=darkrust,citecolor=darkseagreen,urlcolor=darksilver}

\ifpdf
  \DeclareGraphicsExtensions{.eps,.pdf,.png,.jpg}
\else
  \DeclareGraphicsExtensions{.eps}
\fi

\newsiamremark{remark}{Remark}
\newsiamremark{hypothesis}{Hypothesis}
\crefname{hypothesis}{Hypothesis}{Hypotheses}
\newsiamthm{claim}{Claim}

\headers{Global strong solutions of 1-d pressureless IGR}{Ruijia Cao and Florian Sch{\"a}fer}

\title{Information geometric regularization of unidimensional pressureless Euler equations yields global strong solutions} 

\author{Ruijia Cao\thanks{Center for Applied Mathematics, Cornell University 
  (\email{rc948@cornell.edu}).}
\and Florian Sch{\"a}fer\thanks{Courant Institute of Mathematical Sciences, New York University
  (\email{florian.schaefer@nyu.edu}).}}

\usepackage{amsopn}
\usepackage{autonum}

\ifpdf
\hypersetup{
  pdftitle={Information geometric regularization of unidimensional pressureless Euler equations yields global strong solutions},
  pdfauthor={Ruijia Cao and Florian Sch{\"a}fer}
}

\begin{document}

\maketitle

\begin{abstract}
  Partial differential equations describing compressible fluids are prone to the formation of shock singularities, arising from faster upstream fluid particles catching up to slower, downstream ones. 
  In geometric terms, this causes the deformation map to leave the manifold of diffeomorphisms.
  Information geometric regularization addresses this issue by changing the manifold geometry to make it geodesically complete.
  Empirical evidence suggests that this results in smooth solutions without adding artificial viscosity.
  This work makes a first step toward understanding this phenomenon rigorously, in the setting of the unidimensional pressureless Euler equations. 
  It shows that their information geometric regularization has smooth global solutions. 
  By establishing $\Gamma$-convergence of its variational description, it proves convergence of these solutions to entropy solutions of the nominal problem, 
  in the limit of a vanishing regularization parameter.
  A consequence of these results is that manifolds of unidimensional diffeomorphisms with information geometric regularization are geodesically complete.
\end{abstract}

\begin{keywords}
  Compressible flow, shock waves, optimal transport, information geometry, inviscid regularization, geometric hydrodynamics, entropy regularization, pressureless Euler
\end{keywords}

\begin{MSCcodes}
35L65, 49Q22, 49J45, 58B20, 76L05 
\end{MSCcodes}

\section{Introduction}
\subsection{Background: Information Geometric Regularization}
\subsubsection*{Shock formation in the Euler equations}
The barotropic Euler equations 
\begin{equation}
    \label{eqn:euler_velocity}
    \partial_t
    \begin{pmatrix}
         \rho \vct{u} \\
         \rho
    \end{pmatrix}
    + \divergence
    \begin{pmatrix}
         \rho \vct{u} \otimes \vct{u} + P(\rho) \Id \\ 
         \rho \vct{u}
    \end{pmatrix}
    = 
    \begin{pmatrix} 
         \vct{f}\\
         0
    \end{pmatrix},
\end{equation}
for velocity $\vct{u}: \R^d \times \R_+ \to \R^d$ and density $\rho: \R^d \times \R_+ \to \R$ describe the dynamics of an inviscid compressible fluid.
The fluid is assumed to be barotropic, meaning that the pressure $P$ is a function of the density $\rho$ only.
Under most initial conditions, solutions of \cref{eqn:euler_velocity} develop jump discontinuities in $\vct{u}$ and $\rho$, in finite time \cite{sideris1985formation,christodoulou1993global,christodoulou2007formation,christodoulou2014compressible,christodoulou2019shock}.
This is due to faster gas particles inevitably catching up to slower ones, resulting in a steepening of the velocity and density profiles.
Such \emph{shock waves} preclude the existence of strong global solutions to \cref{eqn:euler_velocity} and require specialized numerical treatment. 
\subsubsection*{Euler regularization}
With shocks precluding classical, strong solutions, the canonical solution concepts of Euler equations are weak ``entropy'' solutions that arise as limits of solutions under a vanishingly small regularization of \cref{eqn:euler_velocity}.
A popular regularization amounts to adding a diffusion term to the momentum balance equation, amounting to a Navier-Stokes type regularization.
Especially in the non-barotropic case, thermodynamic considerations suggest modification of the density and energy equations as well \cite{guermond2014viscous}.
For numerical solutions using standard methods, the strength of the regularization needs to be chosen proportionally to the mesh width, requiring highly resolved computational grids to avoid artifacts due to excessive dissipation. 
Numerous proposals for nonlinear viscous regularization attempt to introduce dissipation only near shocks \cite{vonneumann1950method,puppo2004numerical,cook2005hyperviscosity,fiorina2007artificial,mani2009suitability,barter2010shock,guermond2011entropy,bruno2022fc,dolejvsi2003some}.
Nonlinear numerical methods such as MUSCL or WENO lower the approximation order near shocks \cite{van1979towards,ray2018artificial,liu1994weighted,harten1997uniformly,shu1998advanced}, implicitly regularizing the problem through the numerical viscosity of lower-order schemes. 
All these approaches amount to a balancing act of avoiding spurious oscillation due to Gibbs-Runge phenomena without dissipating physically meaningful oscillations associated to sound and entropy waves or turbulence \cite{lele2009shock}.
Inviscid regularizations can potentially circumvent these difficulties.
The earliest example we are aware of is \cite{bhat2005lagrangian} that proposes compressible extensions of the Lagrangian averaged Euler/Navier-Stokes models \cite{holm1998eulera,holm1998eulerb} for shock regularization. 
However, numerical studies of such a regularization in \cite{bhat2006lagrangian} found ``\emph{that it is not well-suited for the approximation of shock solutions of the
compressible Euler equations.}''.
Attempts based on the closely related Leray regularization \cite{leray1934essai} failed to capture the correct shock speeds of the Euler equations, making them unsuitable for most applications \cite{bhat2006hamiltonian,bhat2009riemann,bhat2009regularization}.
The failure is likely due to the indifference of Leray regularization to the variable mass density of compressible flows, 
which causes it to violate the local conservation of momentum.
\cite{guelmame2022hamiltonian} proposes inviscid, nondispersive regularizations in the unidimensional case based on prior work on the shallow water equation \cite{clamond2018non}. 
The modified equation is in conservation form and maintains the correct shock speeds. 
But it lacks a mechanism for dissipating energy in smooth solutions, causing non-physical cusps to form \cite{pu2018weakly,liu2019well}. 
\subsubsection*{Information geometric regularization (IGR)}
Recent work by the authors proposes the first inviscid regularization of the multidimensional Euler equations that produces smooth solutions with the correct shock speed in numerical experiments \cite{cao2023information}.
This information geometric regularization (IGR) is given by 
\begin{equation}
    \label{eqn:reg_euler_intro}
    \begin{cases}
         \partial_t
         \begin{pmatrix}
              \rho \vct{u} \\
              \rho
         \end{pmatrix}
         + \divergence 
         \begin{pmatrix}
              \rho \vct{u} \otimes \vct{u} + \left(P(\rho) + \Sigma\right) \Id\\ 
              \rho \vct{u}
         \end{pmatrix}
         = 
         \begin{pmatrix} 
              \vct{f}\\
              0
         \end{pmatrix} \\
         \rho^{-1} \Sigma - \alpha \divergence(\rho^{-1} \nabla \Sigma) = \alpha  \left(\trace^2\left([\cnst{D} \vct{u}]\right) + \trace\left([\cnst{D} \vct{u}]^2\right) \right).
    \end{cases}
\end{equation}
The parameter $\alpha$ controls the width of the smoothed-out shocks, which is proportional to $\sqrt{\alpha}$.
Equation \cref{eqn:reg_euler_intro} arises from a Lagrangian perspective, in which the Euler equations (up to $P, \vct{f}$) are geodesics on a diffeomorphism manifold \cite{arnold1966geometrie,khesin2021geometric}.
Shock formation amounts to the system reaching the manifold's boundary and thus arises from the geodesic incompleteness of the diffeomorphism manifold.
\cite{cao2023information} avoids singular shocks by replacing $L^2$-geodesics with information geometric \emph{dual} geodesics \cite{amari2016information}, designed to make it geodesically complete.
Expressing the curvature contributions of the modified geometry in Eulerian coordinates yields the entropic pressure $\Sigma$ in \cref{eqn:reg_euler_intro}. 
Adding it to the pressure $P$ allows extending IGR to a wide range of Euler-like equations, including Navier-Stokes, MHD, and non-barotropic gas dynamics. 
Owing to their smoothness, solutions of IGR are amenable to high-order numerical methods on a grid of resolution $\sqrt{\alpha}$ without additional shock capturing techniques, promising to greatly simplify the numerical treatment of flows with shocks.
At the same time, they avoid the spurious dissipation of acoustic waves typical of viscous regularizations \cite[Section 3.3]{barham2025hamiltonian}.
Recently, IGR has enabled the simulation of the compressible Navier-Stokes equations at unprecedented scale, exceeding a quadrillion degrees of freedom \cite{wilfong2025simulating}.
\subsection{This work: Rigorous theory for unidimensional pressureless IGR}
In this work, we make a first step toward a rigorous justification of IGR. 
In the unidimensional case with vanishing external forces $f$ and pressure $P$, we show that the IGR equations \cref{eqn:reg_euler_intro} admit global strong solutions with regularity matching the initial conditions that converge, as $\alpha \rightarrow 0$, to entropy solutions of the nominal system.
\subsubsection*{The unidimensional pressureless case\nopunct} of \cref{eqn:reg_euler_intro} on the interval $[a,b]$, with no external forces, no-penetration boundary conditions, and initial velocity $u_0$ and mass density $\frac{\D \rmeas}{\D \Leb}$ is 
\begin{equation}
    \label{eqn:1d-igr-pressureless}
    \begin{cases}
         \partial_t
         \begin{pmatrix}
              \rho u \\
              \rho
         \end{pmatrix}
         + \partial_x
         \begin{pmatrix}
              \rho u^2 + \Sigma \\ 
              \rho u
         \end{pmatrix}
         = 
         \begin{pmatrix} 
              0\\
              0
         \end{pmatrix}, \quad &\text{for} \quad x \in (a, b), t \in \R_+,\\
         \rho^{-1} \Sigma {- \alpha \partial_x(\rho^{-1} \partial_x \Sigma)} = 2 \alpha \left(\partial_x u\right)^2, \quad &\text{for} \quad x \in (a, b), t \in \R_+,\\
         u(a, t) = u(b, t) = \partial_x \Sigma(a,t) = \partial_x \Sigma(b, t) = 0, \quad &\text{for} \quad t \in \R_+, \\
         u(x, 0) = u_0(x), \quad \rho(x, 0) = \frac{\D \rmeas}{\D \Leb}(x), \quad &\text{for} \quad x \in [a,b],
    \end{cases}
\end{equation}
recovering the unidimensional pressureless Euler equations with $\alpha = 0$. 
We observe that \cref{eqn:1d-igr-pressureless} is invariant under a positive scaling of $\rho$. Thus, we can assume without loss of generality that the initial mass distribution $\rmeas$ is a probability measure ($\int_{a}^{b} \D \rmeas(x) = 1$).
The associated deformation maps $\diff_t: \R \to \R$ satisfy $\dot{\diff}_t = u(\diff_t)$, $\diff_0 = \mathrm{Id}$.
For $\alpha = 0$, prior to shock formation, the path $t \mapsto \diff_t$ is an $L^2$ geodesic on the diffeomorphism manifold of deformation maps $\diff$ \cite{khesin2021geometric}.
Shocks form when $\diff_t$ reaches the manifold's boundary and becomes non-invertible. 
They thus arise from the geodesic incompleteness of diffeomorphism manifolds equipped with the Levi-Civita connection associated to the $L^2$ metric. 
Solutions of \cref{eqn:1d-igr-pressureless} with $\alpha > 0$ instead are \emph{dual geodesics} in the sense of \cite{amari2000methods,amari2016information}, defined by the barrier function $\psi(\diff) \coloneqq \int \left(\diff(x) - x\right)^2 / 2 \D \rmeas + \alpha \int  - \log (\partial_x \diff) \D \rmeas$.
Proving the existence of global smooth solutions of \cref{eqn:1d-igr-pressureless} thus amounts to showing the geodesic completeness of the diffeomorphism manifold equipped with the dual connection induced by $\psi$.
\subsubsection*{Our contributions}
This work provides rigorous justification for the use of information geometric regularization to simplify the numerical treatment of the compressible Euler equations, in the unidimensional, pressureless case.
We prove that under minimal regularity assumptions, \cref{eqn:1d-igr-pressureless} has generalized ``variational'' solutions that are unique, stable with respect to perturbations of the initial conditions, and converge, as $\alpha \rightarrow 0$, to entropy solutions of the nominal system.
For regular initial data, we prove that the IGR equations \cref{eqn:1d-igr-pressureless} admit global strong solutions with regularity matching that of the initial conditions. 
In particular, this implies that unidimensional diffeomorphism manifolds equipped with the dual affine connection induced by the IGR barrier $\bpot$ are geodesically complete.

\subsubsection*{Overview} The remainder of this paper is structured as follows. \textbf{In \cref{sec:existence_minimal_regularity}} we introduce a variational solution concept for IGR that is applicable to $\diff$ that are simply monotone, but not necessarily weakly differentiable. 
Using the direct method of the calculus of variations, we prove the existence of global variational solutions under minimal regularity assumptions on the problem data. 
\textbf{In \cref{sec:convergence_variational}} we prove that the functional defining the variational solutions $\Gamma$-converges, as $\alpha \rightarrow 0$, to the functional defining the entropy solutions of the pressureless Euler equation as studied by \cite{dermoune1999probabilistic,natile2009wasserstein,cavalletti2015simple,hynd2019lagrangian,hynd2020trajectory}.
Thus, we establish convergence of IGR solutions to entropy solutions of the nominal system, in the limit of vanishing regularization.
\textbf{In \cref{sec:weak_solutions}}, we show that the deformation maps of variational solutions are weakly differentiable with derivatives uniformly bounded from above and below.  
Thus, they minimize the IGR barrier function $\bpot$, with a linear perturbation given by the time-scaled initial velocity.  
We use the resulting first order optimality condition and a bootstrapping argument to show that the spatial regularity of the deformation maps matches that of the initial data.
\textbf{In \cref{sec:eulerian}} we derive stability and differentiability of the deformation maps with respect to the initial velocity field, allowing us to take their time derivatives. 
This allows us to conclude that the time-parametrized family of deformation maps obtained in \cref{sec:existence_minimal_regularity} satisfy the Lagrangian form of \cref{eqn:1d-igr-pressureless}.
Taking a second time derivative yields solutions to the Eulerian form, with regularity matching that of the initial data.
\textbf{In \cref{sec:conclusion}} we briefly summarize our results and discuss future directions.
Throughout this work, unless otherwise mentioned, function spaces such as $L^p$ spaces, Sobolev spaces $W^{k,p}$, and spaces $C^k$ of continuously differentiable functions are defined on the interval $[a,b]$.
$\Leb$ denotes the Lebesgue measure.

\section{Existence of variational solutions}
\label{sec:existence_minimal_regularity}
\subsection*{Overview}
This section uses methods from the calculus of variations to establish the existence of solutions of \cref{eqn:1d-igr-pressureless}
under minimal regularity assumptions.
\Cref{sec:igr_dual_geodesics} introduces the Lagrangian form of \cref{eqn:1d-igr-pressureless} and shows that it arises as the dual geodesic equation generated by a convex function on the manifold of diffeomorphisms.
\Cref{sec:variational_solutions} introduces a variational solution concept for \cref{eqn:1d-igr-pressureless} under minimal regularity assumptions.
\Cref{sec:existence_minimizers} proves the existence of minimizers of the associated functional.
\Cref{sec:stability_minimizers} shows that these minimizers are stable under perturbations of the initial data and proves the existence of variational solutions.
\subsection{IGR solutions as dual geodesics}
\label{sec:igr_dual_geodesics}
The unidimensional pressureless IGR equation \cref{eqn:1d-igr-pressureless} arises by regularizing the Lagrangian form of the pressureless Euler equation \cite{cao2023information}, yielding the modified equation\footnote{Note that where \cite{cao2023information} uses arbitrary $\diff_0$ and fixes $\rmeas = \Leb$, we choose here to fix $\diff_0 = \mathrm{Id}$ and allow for arbitrary $\rmeas$.}
\begin{equation}
    \label{eqn:lagrangian_eqn}
    \begin{cases}
        \ddot{\diff}_t\frac{\D \rmeas}{\D \Leb} - \alpha \partial_x\left(\left[\partial_x \diff_t \right]^{-2} [\partial_x (\ddot{\diff}_t)] \frac{\D \rmeas}{\D \Leb}\right) 
        + 2 \alpha \partial_x \left(\left[\partial_x \diff_t \right]^{-3}  [\partial_x \dot{\diff}_t]^2 \frac{\D \rmeas}{\D \Leb}\right) = 0,\\
        \dot{\diff}_0 = u_0(x) 
        \quad \quad \diff_0 = \mathrm{Id},
        \quad \quad \forall t \geq 0: \diff_{t}(a) = a, \quad \diff_{t}(b) = b.
    \end{cases}
\end{equation}
Here, the map $\diff_t$ denotes the flow map at time $t$. 
In particular, the path $t \mapsto \diff_t(x)$ traces the path of a particle starting in $x$ under the modified equation.
The unregularized case ($\alpha = 0$) reduces to Newtonian inertial movement.
IGR modifies the notion of inertial movement such as to prevent particle collisions.

This is the dual geodesic equation on the diffeomorphism manifold, derived from the convex function \cite{cao2023information}
\begin{equation}
\bpot(\diff) \coloneqq \int \limits_{a}^{b} \frac{(\diff(x) - x)^2}{2} - \alpha \log\left(\partial_x \diff(x)\right) \D \rmeas(x).
\end{equation}
As such, it has the form 
\begin{equation}
    \label{eqn:dual_inertial}
    \ddot{\diff} 
+\left[\cnst{D}^2\bpot\left(\diff_t\right)\right]^{-1}\left[\cnst{D}^3\bpot\left(\diff_t\right)\right]\left(\dot{\diff}_t,\dot{\diff}_t\right) = 0
\end{equation}
and is solved by solutions of 
\begin{equation}
    \label{eqn:dual_exp}
    \cnst{D} \psi \left(\diff_t\right)(\cdot) = \cnst{D} \psi(\diff_0)(\cdot) + \left[\cnst{D}^2\psi\left(\diff_0\right)\right]\left(t \dot{\diff}_0, (\cdot)\right).
\end{equation}
Here and in the following, the differentials $\cnst{D}^k\psi$ refer to the $k$-linear form on suitable closures of the space of smooth compactly supported vector fields, defined by taking successive directional derivatives. 
Dual geodesics play a key role in information geometry \cite[Sections 1.5 and 6.5 ]{amari2016information}, where they arise as a natural notion of ``straight lines'' on manifolds of probability distributions. 
The most canonical example is the probability simplex $\left\{p \in \R^n: p_i \ge 0, \sum_{i=1}^n p_i = 1\right\}$, equipped with the discrete negative Shannon entropy $\psi(p) = \sum_{i=1}^n p_i \log(p_i)$. 
Dual geodesics arise from the unique torsion-free affine connection that is invariant under sufficient statistics, flat, and geodesically complete \cite{fujiwara2024hommage}. 
The dual geodesic \cref{eqn:dual_inertial} used by information geometric regularization amounts to interpreting $\diff$ as the parameter of a statistical model and $\bpot(\diff)$ as the associated entropy.  
This is justified by viewing $\diff$ as a transport map pushing forward the reference measure $\rmeas$ to the mass distribution $\rho = (\diff)_{\#} \rmeas$. 
For Lebesgue absolutely continuous $\rmeas$, a change of variables yields
\begin{align}
  \label{eqn:entropy_pushforward}
  - \int \limits_{a}^{b} \log \left(\partial_x \diff \right) \D \rmeas 
  = \int \limits_a^b \log \left( \frac{\dd \rho}{\dd \Leb} \circ \diff \right) \rho \circ \diff \cdot \partial_x \diff \D \Leb - \int \limits_a^b \log \left( \frac{\dd \rmeas}{\dd \Leb} \right) \D \rmeas 
  &= \Ent \left( \diff_{\#} \mu \right) - \Ent \left(\mu \right), 
\end{align}
for the continuous negative entropy $\Ent\left(\nu\right) \coloneqq \int \frac{\dd \nu}{\dd \Leb} \log \left(\frac{\dd \nu}{\dd \Leb}\right) \D \Leb$. 
Dual geodesics obtained from the sum of the entropy of the pushforward measure and the squared Euclidean norm of the transport are a compromise between information geometric dual geodesics and $L^2$-Wasserstein geodesics. 
For small $\alpha$, they mostly behave like the Wasserstein geodesics, giving rise to the regular pressureless Euler equation. 
But as $\rho$ approaches singularity due to shock formation, the entropic term dominates and prevents shock singularities from forming.
\begin{remark} [Dual geodesics are not metric]
  Different from most geodesics and affine connections considered in the literature, dual geodesics are not metric, in the sense that dual parallel transport of two vectors $u$, $v$ does generally not conserve a Riemannian metric. 
  Instead, parallel transporting $u$ according to the Euclidean connection and $v$ according to the dual connection induced by $\bpot$ preserves the Hessian metric $\cnst{D}^2\bpot\left(u, v\right)$ \cite[Section 6.1]{amari2016information}. 
  Thus, where metric geodesics $t \mapsto \diff_t$ conserve the energy $\langle \dot{\diff}, \dot{\diff} \rangle_{[\cnst{D}\bpot\left(\diff_t\right)]}$, dual geodesics (and thus, IGR solutions) do not generally conserve an energy.
  This property is crucial, as it allows for energy dissipation in smooth solutions and thus the existence of global smooth solutions that mimic the irreversible behavior of shock waves.
  In contrast, Hamiltonian regularizations such as \cite{guelmame2022hamiltonian} conserve a modified energy ($\int \rho u^2 + \alpha \rho (\partial_x u)^2 \D x$ in the pressureless case). 
  We believe that this is the reason why these regularizations do not have globally smooth solutions. 
  Instead of dissipating energy in smooth solutions, they need to ``hide'' the initial energy in the Dirichlet term of the regularized Hamiltonian by forming cusp singularities \cite{pu2018weakly,liu2019well}.
\end{remark}
\subsection{Variational solutions}
\label{sec:variational_solutions}
Homotopy interior point methods minimize a linear functional under a convex constraint by regularizing it with a barrier function that blows up at the boundary of the constraint set.
Solving the resulting unconstrained optimization problem for decreasing weight $1/t$ of the barrier yields a so-called central path that converges to the solution of the original problem \cite{nemirovski2008interior}.
As observed in \cite[Section 13.5]{amari2016information}, the dual geodesics associated to a convex function $\bpot$ are the central paths of the homotopy interior point method that uses $\bpot$ (or, more generally, the associated Bregman divergence) as a barrier function.
This motivates the observation that Equation \eqref{eqn:dual_exp} is the optimality criterion of the minimization
\begin{equation}
    \label{eqn:variational_exp}
    \min \limits_{\diff \text{ monotone }} \bpot(\diff) - [\cnst{D} \psi(\diff_0)](\diff) - t \left[\cnst{D}^2\psi\left(\diff_0\right)\right]\left(\dot{\diff}_0, \diff\right), 
\end{equation}
motivating us to define, for an arbitrary linear functional $f$ and probability measure $\rmeas$, the functional
\begin{equation}
    \hat{F}_{\alpha}^{f, \rmeas} = 
    \begin{cases}
        \int \limits_{a}^{b} \frac{(\diff(x) - x)^2}{2} - \alpha \log\left(\partial_x \diff\right) \D \rmeas(x) - f(\diff), & \text{if } \diff \text{ is monotone},\\
        \hfill \infty, \hfill & \text{else}.
    \end{cases}
\end{equation}
We assume that $\rmeas$ and the Lebesgue measure $\mathscr{L}$ are absolutely continuous with respect to each other.
We account for $t$-dependence of $\diff_t$ by changing $f$ with $t$, $\hat{F}$ itself is independent of $t$.
For a function $v \in L^1(\rmeas)$, we write $\hat{F}_{\alpha}^{v,\rmeas} \coloneqq \hat{F}_{\alpha}^{u \mapsto \int u v  \D \rmeas(x), \rmeas}$.
We study the existence of solutions to \cref{eqn:1d-igr-pressureless} through this variational principle.
The convexity of $\hat{F}_{\alpha}^{f, \rmeas}$ suggests using the direct method of the calculus of variations \cite{dacorogna2007direct}. 
But the function $\log(\partial_x \diff)$ does not play well with any function space that we are aware of. 
A topology rendering it lower-semicontinuous seemingly needs to control weak derivatives, but due to the slow growth of the logarithm, it does not yield coercivity in any Sobolev space $W^{k, p}$ with $p > 1$. 
Sobolev spaces with $p = 1$ or the log-Orlicz-Sobolev space, on the other hand, are not reflexive \cite{clason2021entropic} and thus their weak-* compactness does not imply the weak compactness required by the direct method.
We circumvent this problem by using the monotonicity of $\diff$.
For sufficiently regular $\rmeas, \diff$, a change of variables as in Equation~\cref{eqn:entropy_pushforward}  reveals that the functional $\hat{F}_{\alpha}^{f, \rmeas}$ is equal, up to a constant ($\alpha \cdot \Ent(\rmeas)$), to the relaxed functional 
\begin{equation}
  \label{eqn:relaxed_functional}
    F_{\alpha}^{f, \rmeas} = 
    \begin{cases}
        \int \limits_{a}^{b} \frac{(\diff(x) - x)^2}{2} \D \rmeas(x) + \alpha \cdot \Ent(\diff_{\#} \rmeas)   - f(\diff) , & \text{if } \diff \text{ is monotone},\\
        \hfill \infty, \hfill & \text{else}.
    \end{cases}
\end{equation}
For a function $v \in L^1(\rmeas)$ we write $F_{\alpha}^{v, \rmeas} \coloneqq F_{\alpha}^{u \mapsto \int u v  \D \rmeas(x), \rmeas}$ as before.
Here, $\Ent(\nu) $ denotes the negative Shannon entropy of $\nu$, defined by
\begin{align}
  \Ent\left(\nu\right)
  =
  \begin{cases}
    \int \frac{\D \nu}{\D \Leb} \log \left( \frac{\D \nu}{\D \Leb} \right) \D \Leb & \text{if } \nu \ll \Leb, \\
    \infty & \text{otherwise}.
  \end{cases}
\end{align}
The pushforward $\diff_{\#} \rmeas$ is defined by $\diff_{\#} \rmeas(A) = \rmeas(\diff^{-1}(A))$ for any Borel measurable set $A$.
Casting the problem in this form allows us to first establish the existence of monotone (but not necessarily weakly differentiable) minimizers using the excellent weak compactness properties of measures and monotone functions.  
Minimizers of $F_{\alpha}^{f, \rmeas}$, for $f$ encoding the initial density and velocity, are the weakest notion of solution of \eqref{eqn:1d-igr-pressureless} that we consider in this work.
\begin{definition}[Variational solution]
  \label{def:variational_solution}
  Let $\Probs([a,b])$ denote the set of probability measures on $[a, b]$. For $\rmeas \in \Probs([a, b])$ and measurable initial data $u_0 \colon [a,b] \to \R$, define the functional
  \begin{equation}
    f_{\alpha}^{u_0}(\diff) \coloneqq \int \limits_{a}^{b} \diff \cdot u_0 + \alpha \partial_x \diff \cdot \left(\partial_x u_0 \right) \D \rmeas(x) 
    - \int \limits_{a}^{b} \alpha \partial_x \diff \D \rmeas(x).
  \end{equation}
  A \emph{variational solution} of \eqref{eqn:1d-igr-pressureless} with initial data $u_0$ is a pair $u: [a,b] \times \R_+ \to \R$ and $\rho: \R_+ \to \Probs([a,b])$, such that for $\diff_t$ a monotone minimizer of $F_{\alpha}^{f_{\alpha}^{t u_0}, \rmeas}$, and $t \mapsto \diff_t$ weakly differentiable as a map from $t$ to $L^2(\rmeas)$, 
  \begin{align}
    \label{eqn:eulerian_sticky_particles}
    \rho_t &= \left(\diff_{t}\right)_{\#} \rmeas, \\
    u(y, t) &= 
    \begin{cases}
        \dot{\diff}_t( \diff_t^{-1}(y)), \hfill &\text{if $\# \diff_t^{-1}(y) = 1$},\\
        0, \hfill &\text{if $\#\diff_t^{-1}(y) = 0$}\\
        \frac{1}{\rmeas \left(\diff_t^{-1}(y) \right)} \int \limits_{\diff_t^{-1}\left(y\right)} \dot{\diff}_t(x) \D \rmeas(x), \hfill &\text{else.}
    \end{cases}
  \end{align}
  It is a \emph{regularized} variational solution if $\diff_t$ instead minimizes $F_{\alpha}^{t u_0 + f_{\alpha}^{0}, \rmeas}$.
  Here, $\#(\cdot)$ denotes set cardinality.
\end{definition}
\begin{lemma}
  With this definition, $u(\cdot, t)$ is in $L^2(\rho_t)$, for all $t \ge 0$ outside a set of measure zero. The function $u(\cdot, t)$ does not depend on the choice of $\dot{\diff}_t$ in the equivalence class up to $\rmeas$-a.e. equality.
\end{lemma}
\begin{proof}
  For measurable $v\colon [a, b] \to \R$, we have by the change of variables theorem \cite[Theorem 3.6.1]{bogachev2007measure}, that  
  \begin{equation}
    \int \limits_{a}^{b} u(y, t) \cdot v(y) \D \rho_t(y) = \int \limits_{a}^{b} u(\diff_t(y), t) \cdot v(\diff_t(y)) \D \rmeas = \int \limits_{a}^{b} \dot{\diff}_{t} v(\diff_t(y)) \D \rmeas.
  \end{equation}
  Setting $v = u(\cdot, t)$ shows that $u(\cdot, t) \in L^2(\rho_t)$, so long as $\dot{\diff}_t$ is in $L^2(\rmeas)$. 
  Since $t \to \dot{\diff}_t$ is the weak derivative of $t$ to $L^2(\rmeas)$, it is in $L^1_{\mathrm{loc}}(\R_+; L^2(\rmeas))$. 
  Thus $\dot{\diff}_t \in L^2(\rmeas)$ for almost all $t$.
  We observe that the right-hand side of the equation is independent of the choice of $\dot{\diff}_t$ in the equivalence class up to $\rmeas$-a.e. equality.
\end{proof}
\begin{remark}
  The motivation for introducing regularized variational solutions is that many applications feature initial value problems with discontinuous $\frac{\D \rmeas}{\D \Leb}, u_0$.  
  IGR equations may not have solutions in this case, requiring smoothing of the initial data \cite{cao2023information}. 
  This smoothing is implicit in the definition of regularized variational solutions, since the initial velocity $u_0$ is replaced by the solution $w$ of the elliptic problem $f_{\alpha}^{w} = u_0$.
\end{remark}
\begin{remark}
  The variational problems $F$ and $\hat{F}$ are reminiscent of prior works regularizing  optimal transport using Dirichlet energies \cite{de2016monge} or entropic terms \cite{cuturi2013sinkhorn,clason2021entropic}.
  A key difference is that our variational problem describes the exponential map for an initial velocity and regularizes the end point $\rho$ (note that the regularization depends on $\diff$ only through its pushforward). 
  Instead, \cite{de2016monge,cuturi2013sinkhorn,clason2021entropic} consider the path between two fixed end points and regularize the transport plan between them. 
\end{remark}
\subsection{Existence of minimizers}
\label{sec:existence_minimizers}
We prove the existence of minimizers of $F_{\alpha}^{f, \rmeas}$ in the Helly space 
\begin{equation}
  \helly \coloneqq \left\{ \diff\colon [a, b] \to [a, b]: \diff \text{ monotone and } \diff(a) = a, \diff(b) = b\right\}.
\end{equation}
We begin by recalling basic topological properties of $\helly$.
\begin{lemma}\label[lemma]{le:helly_properties}
  The Helly space $\helly$, equipped with the topology of point-wise convergence, is compact and first countable (thus, compactness implies sequential compactness).
  Its elements are Borel measurable.
\end{lemma}
\begin{proof}
  See \cite[Example 107 and Example 105]{steen1978counterexamples}.
\end{proof}
\begin{corollary} \label{le:seq_comp_mono_measure}
  Let $\nu$ be a probability measure on $[a, b]$ and $(\diff_n)_{n \ge 1}$ any sequence in $\helly$.
  Then there exist $\diff \in \helly$ and a subsequence $(\diff_{n_k})_{k \ge 1}$ such that $\diff_{n_k} \to \diff$ pointwise on $[a, b]$ and $(\diff_{n_k})_{\#} \nu \to \diff_{\#} \nu$ weakly.
\end{corollary}
\begin{proof}
  Write $\mu_n = (\diff_{n})_{\#} \nu$.
  By \cref{le:helly_properties}, there exists a subsequence $(\diff_{n_k})_{k\ge 1}$ such that $\diff_{n_k} \to \diff$ pointwise on $[a, b]$. 
 Let $\mu = \diff_{\#} \nu$.
 Then, we claim that $\mu_{n_k} \to \mu$ weakly as $k \to \infty$.
 Indeed, for any $f \in C_b([a, b])$,
 \begin{align}
   \lim_{k \to \infty} \int f \D \mu_{n_k} 
   &= \lim_{k \to \infty} \int f \D (\diff_{n_k} )_{\#} \nu \\ 
   &= \lim_{k \to \infty} \int f\circ \diff_{n_k} \D \nu \quad \text{Change of Variable} \\
   &= \int f \circ \diff \D \nu \quad \text{Bounded Convergence Theorem} \\
   &= \int f \D (\diff)_{\#} \nu \\
   &= \int f \D \mu.
 \end{align}
 Therefore, $\mu_{n_k} \to \mu$ weakly.
\end{proof}
Weak compactness of $\helly$ and lower semicontinuity of $\Ent\left(\cdot\right)$ imply the existence of minimizers of $F_{\alpha}^{f, \rmeas}$.
\begin{theorem}[Minimization of $F^{\ugen, \rmeas}_\alpha$]\label{thm:minimizer_existence}
  For $\ugen \in L^1(\rmeas)$, $\mu \ll \mathscr{L}$ such that 
  \begin{equation}
    0 < \essinf_{x \in [a, b]} \mu(x) \le \esssup_{x \in [a, b]} \mu(x) < \infty,
  \end{equation}
  there exists a unique $\diff^* \in \helly$ such that 
  \begin{equation}
    F_{\alpha}^{\ugen, \rmeas}(\diff^*) = \inf_{\diff \in \helly} F_{\alpha}^{\ugen, \rmeas}(\diff).
  \end{equation}
\end{theorem}
\begin{proof}
  First, for any $\diff \in \helly$, Jensen's inequality implies that $\Ent\left(\diff_{\#} \rmeas \right) > -\infty$.
  Since $a \le \diff \le b$ for all $\diff \in \helly$, we see that $\inf_{\diff \in \helly} F_{\alpha}^{\ugen, \rmeas}(\diff) > -\infty$.
  Furthermore, since $\mathrm{Id} \in \helly$, we see that $\inf_{\diff \in \helly} F_{\alpha}^{\ugen, \rmeas}(\diff) < \infty$.
  Hence, there exists $\gamma \in \R$ such that $\gamma = \inf_{\diff \in \helly} F_{\alpha}^{\ugen, \rmeas}(\diff)$.
  Let $(\diff_k)_{k \ge 1} \subseteq \helly$ be a minimizing sequence. 
  By \cref{le:seq_comp_mono_measure}, there exist a subsequence $(\diff_{n_k})_{k \ge 1}$ and
  $\diff^* \in \helly$ such that $\diff_{n_{k}} \longrightarrow \diff^*$ pointwise and $(\diff_{n_k})_{\#} \rmeas \longrightarrow \diff^*_{\#} \rmeas$.
  Therefore, by Fatou's lemma and weak lower semi-continuity of $\Ent(\cdot)$ (see \cite[Theorem 1]{posner1975random}, for example), we have
  \begin{align}
   \gamma 
   &= \liminf_{k \to \infty} 
   \frac{1}{2} \int \left(\diff_{n_k}(x) - \mathrm{Id} \right)^2 
   - \diff_{n_k}(x) \cdot  \ugen (x)  \D \rmeas(x)  + \alpha \cdot \Ent\left( (\diff_{n_k})_{\#} \rmeas \right)  \\
   & \ge 
   \frac{1}{2} \int \left(\diff^*(x) - \mathrm{Id}\right)^2 
   - \diff^*(x) \cdot  \ugen (x)  \D \rmeas(x) + \alpha \cdot \Ent\left( (\diff^*)_{\#} \rmeas \right).
  \end{align}
  By definition of the infimum, we have
  \begin{equation}
    \gamma \le 
   \int_a^b \frac{1}{2} \left(\diff^*(x) - \mathrm{Id} \right)^2 
   - \diff^*(x) \cdot \ugen (x)  \D \rmeas(x)
    + \alpha \cdot \Ent\left(\left(\diff^*\right)_{\#} \rmeas\right) 
  \end{equation}
  and thus
  \begin{equation}
    \gamma 
    = 
   \frac{1}{2} \int \limits_a^b \left(\diff^*(x) - \mathrm{Id} \right)^2 
   - \diff^*(x) \cdot  \ugen (x)  \D \rmeas(x)
   + \alpha \cdot \Ent \left( \diff^*_{\#} \rmeas \right).
  \end{equation}
  Hence, the existence result follows. 
  The uniqueness follows from \cref{thm:uniqueness_minimizers}.
\end{proof}
\begin{lemma}\label[lemma]{le:injective_minimizer}
  Suppose that $\ugen \in L^1(\rmeas)$, $\rmeas \ll \mathscr{L}$, $\alpha > 0$, $\Ent(\rmeas) < \infty$ and $\frac{\D \rmeas}{\D \Leb} > 0$.
  For $\diff^*$ minimizer of $F_{\alpha}^{\ugen, \rmeas}$, 
  \begin{enumerate}[label=(\roman*)]
    \item We have $\left(\diff^*\right)_{\#} \rmeas \ll \Leb$.
     We write its Radon-Nikodym derivative as $\rho \defeq \frac{\D \diff^*_{\#} \rmeas}{\D x}$
      \label{it:pushforward_is_ac}
    \item $\diff^*$ is almost everywhere differentiable.\label{it:derivative_exists_ae}
    \item $\diff^*$ is injective and the density $\rho$ of $\left(\diff^*\right)_{\#} \rmeas$ with respect to $\Leb$ satisfies a change of variable formula, 
    i.e., $\partial_x \diff^*(x) \cdot \rho \circ \diff^*(x) = \frac{\D \rmeas}{\D \Leb} $ a.e. on $[a, b]$ for $\partial_x \diff^*$ the a.e. derivative of $\diff^*$. 
    Furthermore, $\partial_x \diff^* > 0$ a.e. and we have $\rho \geq 0$ and $\rho(x) > 0$ (only) on $x \in \diff^*([a, b])$. 
    \label{it:change_of_vars}
  \end{enumerate}
  The same result holds for convex combinations of minimizers of $F^{\ugen, \rmeas}_{\alpha}$ for different $\ugen \in L^1(\rmeas)$.
\end{lemma}
\begin{proof}
    Let $\diff^* \in \helly$ be any minimizer of $F_{\alpha}^{\ugen, \rmeas}$. 
    To establish (i), observe that $F_{\alpha}^{\ugen, \rmeas}(\diff^*) \leq F_{\alpha}^{\ugen, \rmeas}(\mathrm{Id}) < \infty$ and the pushforward measure $\diff^*_{\#}\rmeas$ is supported on $\diff^* ([a, b])$.
    Since $F_{\alpha}^{\ugen, \rmeas}(\diff^*)$ would otherwise be infinite, $\diff^*_{\#}\rmeas$ is absolutely continuous with respect to the Lebesgue measure.
    
    To establish (ii), observe that $\diff^*$ is monotone and hence differentiable almost everywhere \cite[Theorem 5.4.2]{heil2019introduction}.

    To establish (iii), we first claim that $\diff^*$ is strictly increasing. 
    Toward a contradiction, suppose that there exist $a \le x_1 < x_2 \le b$
    such that $\diff^*(x_1) = \diff^* (x_2)$.
    Then, since $(\diff^*)_{\#} \rmeas = \rho \D x$, we see that 
    \begin{equation}
      \int \limits_{[x_1, x_2]} \mu(x) \D x 
      = \rmeas [x_1, x_2] 
      = \int \limits_{[\diff^*(x_1), \diff^*(x_2)]}\rho(x) \D x 
      = 0,
    \end{equation}
    contradicting the fact that $\mu(I) > 0$ for all $I \subseteq [a, b]$ of positive length.
    Therefore, $\diff^*$ is strictly increasing and injective.
    Finally, as $\diff^*$ has bounded variation and is injective on $[a, b]$, we may apply \cite[Theorem 11.1]{villani2009optimal} to conclude that $\partial_x \diff^*(x) \cdot \rho \circ \diff^*(x) = \frac{\D \rmeas}{\D \Leb} $ a.e. on $[a, b]$.
    By that same theorem, $\left(\diff^*\right)_{\#} \rmeas \ll \Leb$ implies that $\partial_x \diff^* > 0$ a.e.
    By the same theorem, the Borel function $\rho \colon [a, b] \to \R$ is non-negative, positive only on $\diff^*([a, b])$, 
    satisfying 
    $\D (\diff^*)_{\#} \rmeas = \rho(x)  \D x$.

    We now extend the result to the strict convex combination $\diff = (1 - s)\, \diff^{v_1} + s\, \diff^{v_2}, s \in (0, 1)$ of two minimizers $\diff^{v_1}, \diff^{v_2}$ of $F_{\alpha}^{v_1, \rmeas}$ and $F_{\alpha}^{v_2, \rmeas}$, respectively. 
    The membership in $\Helly$, strict monotonicity (hence, injectivity), and a.e. differentiability of part \ref{it:derivative_exists_ae} follow directly. 
    It remains to show that $\diff_{\#}\rmeas \ll \Leb$.

    For $a \le x < y \le b$, monotonicity of $\diff^{v_2}$ gives
    \begin{equation}
      \diff(y) - \diff(x)
      = (1 - s)\big(\diff^{v_1}(y) - \diff^{v_1}(x)\big) + s\big(\diff^{v_2}(y) - \diff^{v_2}(x)\big)
      \ge (1 - s)\big(\diff^{v_1}(y) - \diff^{v_1}(x)\big).
    \end{equation}
    Since $\diff$ is injective, $g \defeq \diff^{v_1} \circ \diff^{-1}$ is well defined on $\diff([a, b])$, and the bound above reads
    $0 \le g(\diff(y)) - g(\diff(x)) = \diff^{v_1}(y) - \diff^{v_1}(x) \le \tfrac{1}{1 - s}\big(\diff(y) - \diff(x)\big)$;
    that is, $g$ is $\tfrac{1}{1 - s}$-Lipschitz.
    Consequently, for any Lebesgue-null set $N$, the set $\diff^{v_1}\big(\diff^{-1}(N)\big) = g\big(N \cap \diff([a, b])\big)$ is Lebesgue-null.
    Injectivity of $\diff^{v_1}$ and $\diff^{v_1}_{\#} \rmeas \ll \Leb$ yields $\diff_{\#} \rmeas(N) = \rmeas\big(\diff^{-1}(N)\big) = \diff^{v_1}_{\#} \rmeas\big(\diff^{v_1}(\diff^{-1}(N))\big) = 0$ and hence  $\diff_{\#} \rmeas \ll \Leb$.
    Together with strict monotonicity and bounded variation, \cite[Theorem 11.1]{villani2009optimal} then yields the change-of-variables formula (iii) as above.
\end{proof}
\subsection{Stability of minimizers, existence of variational solutions}
\label{sec:stability_minimizers}
In order to prove the uniqueness and stability of minimizers, we will use the following lemma.
\begin{lemma}
  \label[lemma]{lem:lipschitz_cont_solution_map}
  Let $\mathcal{H}$ be a Hilbert space and let $F\colon \mathcal{H} \to \overline{\R} \coloneqq \R \cup \{\infty\}$ with $F(\mathcal{H}) \not \subset \{\infty\}$. 
  Let $u, v \in \mathcal{H}$ and, for $w \in \{u, v\}$, $\diff^w \in \argmin_{\diff \in \mathcal{H}} F(\diff) - \langle w, \diff \rangle_{\mathcal{H}}$ (in particular, the infimum is attained). 
  Let $F$ be $\gamma$-strongly convex on the convex hull of $\diff^u$ and $\diff^v$, the latter meaning that $\diff \mapsto F(\diff) - \frac{\gamma}{2} \langle \diff, \diff\rangle_{\mathcal{H}}$ is convex.
  Then, $\|\diff^u - \diff^v\|_{\mathcal{H}} \leq \gamma^{-1}\|u-v\|_{\mathcal{H}}$.
\end{lemma}
\begin{proof}
  Define $f(s) = F\left(\diff^u + s\left(\diff^v - \diff^u\right)\right)$ and note that $f(0), f(1) < \infty$, since $F(\mathcal{H}) \not \subset \{\infty\}$ and by optimality of $\diff^u,\diff^v$. 
  Then, the $\gamma$-strong convexity of $F$ implies that
  \begin{equation}
    f(s) \leq f(0) + s \left(f(1) - f(0)\right) - \tfrac{\gamma}{2}s(1-s)\|\diff^u - \diff^v\|_{\mathcal{H}}^2.
  \end{equation}
    By optimality of $\diff^u$, we have $f(s) -  \left\langle u, \diff^u + s(\diff^v - \diff^u) \right\rangle_{\mathcal{H}} \geq f(0) - \left\langle u, \diff^u \right\rangle_{\mathcal{H}}$.
  Rearranging terms, we obtain $f(s) \geq  f(0) + s \left\langle u, \diff^v - \diff^u\right\rangle_{\mathcal{H}}$
  and thus $s \left\langle u, \diff^v - \diff^u\right\rangle_{\mathcal{H}} \leq s \left(f(1) - f(0)\right) - \tfrac{\gamma}{2}s(1-s)\|\diff^u - \diff^v\|_{\mathcal{H}}^2$.
  Dividing by $s$ and taking its limit to zero yields
  $\left\langle u, \diff^v - \diff^u\right\rangle_{\mathcal{H}} \leq \left(f(1) - f(0)\right) - \tfrac{\gamma}{2}\|\diff^u - \diff^v\|_{\mathcal{H}}^2$.
  Analogously, optimality of $\diff^v$ implies that $f(s) -  \left\langle v, \diff^v + (1 - s)(\diff^u - \diff^v) \right\rangle_{\mathcal{H}} \geq f(1) - \left\langle v, \diff^v \right\rangle_{\mathcal{H}}$.
  Rearranging terms, we obtain $f(s) \geq  f(1) + (1 - s) \left\langle v, \diff^u - \diff^v\right\rangle_{\mathcal{H}}$ and $(1 - s) \left\langle v, \diff^u - \diff^v\right\rangle_{\mathcal{H}} \leq (1 - s) \left(f(0) - f(1)\right) - \tfrac{\gamma}{2}s(1-s)\|\diff^u - \diff^v\|_{\mathcal{H}}^2$.
  Dividing by $(1 - s)$ and letting $s$ go to $1$ yields $\left\langle v, \diff^u - \diff^v\right\rangle_{\mathcal{H}} \leq \left(f(0) - f(1)\right) - \tfrac{\gamma}{2}\|\diff^u - \diff^v\|_{\mathcal{H}}^2$.
  Adding up the two results, the finite values $f(0), f(1)$ cancel and we obtain
  $\gamma \|\diff^u - \diff^v\|_{\mathcal{H}}^2 \leq \left\langle u - v, \diff^u - \diff^v\right\rangle_{\mathcal{H}}$.
  Finally, the Cauchy-Schwarz inequality implies that $\gamma \|\diff^u - \diff^v\|_{\mathcal{H}}^2 \leq \|u - v\|_{\mathcal{H}} \|\diff^u - \diff^v\|_{\mathcal{H}}$, yielding the result.
\end{proof}
\begin{theorem}[Uniqueness and stability of minimizers]
  \label{thm:uniqueness_minimizers}
  Let $u, v \in L^1(\rmeas)$ with $(u - v) \in L^2(\rmeas)$, $\rmeas \ll \mathscr{L}$ with a.e. positive density and $\Ent(\rmeas) < \infty$. 
  Let $\diff^u, \diff^v$ be minimizers of $F_{\alpha}^{u, \rmeas}$ and $F_{\alpha}^{v, \rmeas}$, respectively. 
  Then,
  \begin{equation}
    \left\| \diff^u - \diff^v \right\|_{L^2(\rmeas)} \leq \left\| u - v \right\|_{L^2(\rmeas)}.
  \end{equation}
  In particular, minimizers of $F_{\alpha}^{u, \rmeas}$ are unique.
\end{theorem}
\begin{proof}
  The result is vacuously true if $\|u - v\|_{L^2(\rmeas)} = \infty$, we thus assume it to be finite. 
  We first show that for all $\diff$ in the convex hull of $\diff^u$ and $\diff^v$
  \begin{equation}
    \alpha \Ent \left( \diff_{\#} \rmeas \right) 
    = \alpha \int \limits_a^b - \log (\aederiv_x \diff) \D \rmeas(x) + \alpha \Ent\left(\rmeas\right). 
  \end{equation}
  Here, $\aederiv_x \diff$ denotes the derivative on a set of full measure on $[a, b]$.
  It exists by monotonicity of $\diff$, even though $\diff$ may not be absolutely continuous or weakly differentiable. 
  In particular, this implies convexity of $\alpha \Ent\left((\cdot)_{\#} \rmeas \right)$ on the convex hull of $\diff^u$ and $\diff^v$.
  Generalizations of this statement are known as displacement convexity of entropies \cite{cavalletti2015simple}. 
  We prove the case needed here, to keep the paper self-contained.

 The case $\alpha = 0$ is trivial, so we assume $\alpha > 0$.
 Using the definition of the pushforward measure and the Change of Variable Theorem \cite[Theorem 3.6.1]{bogachev2007measure}, we have
  \begin{align}
    \alpha \Ent \left( \diff_{\#} \rmeas \right) 
    = \alpha \int \limits_{\diff([a, b])} \log \rho(x) \D \diff_{\#} \rmeas
    = \alpha  \int \limits_{[a, b]}  \log \rho \circ \diff (x) \D \rmeas(x),
  \end{align}
  where $\rho$ is defined as in \cref{le:injective_minimizer}. Now, by \cref{le:injective_minimizer} applied to $\diff$,
 we have a.e.
 \begin{align}
   \aederiv_x \diff(x) \rho \circ \diff(x) 
   &= \dd \rmeas/\dd \Leb(x).
 \end{align}
 By \cref{le:injective_minimizer}, we have $\aederiv_x \diff > 0$ a.e. for any element $\diff$ of the convex hull of $\diff^u, \diff^v$, and can write
 \begin{equation}
   \rho \circ \diff(x) 
   = \frac{\dd \rmeas/\dd \Leb(x)}{\aederiv_x \diff(x)} 
 \end{equation}
 a.e. on $[a, b]$ and hence
  \begin{equation}
    \alpha \int \limits_{a}^{b} \log \rho \circ \diff(x) \D \rmeas(x) 
    = 
    \alpha \int\limits_{a}^{b} - \log \aederiv_x \diff \D \rmeas(x) + \alpha \Ent\left(\rmeas\right).
  \end{equation}
  Thus, we see that 
  \begin{align}
    \alpha \Ent \left( \diff_{\#} \rmeas \right) 
    = \alpha \int \limits_a^b - \log (\aederiv_x \diff) \D \rmeas(x) + \alpha \Ent\left(\rmeas\right).
  \end{align}
  By linearity of $\aederiv$, the above functional is convex in $\diff$. 
  Hence, the map 
  \begin{equation}
    \diff \mapsto \int_{a}^{b} \frac{(\diff(x) - x)^2}{2} \D \rmeas + \alpha \Ent \left( \diff_{\#} \rmeas \right)
    - \frac{1}{2} \int \limits_a^b \diff(x)^2 \D \rmeas
  \end{equation}
  is convex, implying also that $\diff \mapsto \int_{a}^{b} \frac{(\diff(x) - x)^2}{2} \D \rmeas + \alpha \Ent \left( \diff_{\#} \rmeas \right)$ is 1-strongly convex on the convex hull of $\diff^u, \diff^v$ with respect to $L^2(\rmeas)$.
  The result follows by applying \cref{lem:lipschitz_cont_solution_map} to the Hilbert space $L^2(\rmeas)$, the functional $F_\alpha^{u,\rmeas}$ (strong convexity is not affected by adding affine functions), and $w \in \{0, v - u\} \subset L^2(\rmeas)$.
\end{proof}
\begin{theorem}[Well-posedness of variational solutions (\cref{def:variational_solution})]
  For a fixed probability measure $\rmeas \ll \mathscr{L}$ with $0 < \essinf \frac{\D \rmeas}{\D \Leb}, \esssup \frac{\D \rmeas}{\D \Leb} < \infty$, let $ \uvel_0 \frac{\D \rmeas}{\D \Leb} - \alpha \partial_x\big( \left(\partial_x u_0\right) \frac{\D \rmeas}{\D \Leb}\big)$ [or $\uvel_0 \frac{\D \rmeas}{\D \Leb}$] be in $L^2$ and $\partial_x \left(\frac{\D \rmeas}{\D \Leb}\right)$ in $L^1$. 
  Then, there exists a unique [regularized] variational solution with initial velocity $u_0$. 
  The map $u_0 \frac{\D \rmeas}{\D \Leb} - \alpha \partial_x\big( \partial_x (u_0) \frac{\D \rmeas}{\D \Leb}\big) \mapsto \diff_t$  [respectively $u_0  \mapsto \diff_t$] is continuous with respect to the $L^2$ topology.
\end{theorem}
\begin{proof}
  The existence of minimizers of the functionals $F_{\alpha}^{f_{\alpha}^{t u_0}, \rmeas}$ in \cref{def:variational_solution} is guaranteed by \cref{thm:minimizer_existence}.
  By \cref{thm:uniqueness_minimizers}, the minimizers are unique and $L^2$-Lipschitz continuous with respect to $f^{t u_0}_{\alpha}$, establishing Lipschitz continuity of the input-solution map. 
  Since Lipschitz continuous maps are weakly differentiable, we can thus define the Eulerian solution as in \cref{def:variational_solution}.
\end{proof}

\section{Convergence of variational solutions}
\label{sec:convergence_variational}
\subsection*{Overview}
This section studies the convergence of solutions of the regularized problem to solutions of the nominal problem as $\alpha \to 0$.
\Cref{sec:nominal_problem} summarizes existing results for the nominal pressureless Euler equations ($\alpha = 0$).
\Cref{sec:gamma_convergence} uses $\Gamma$-convergence to show that variational solutions of the regularized problem converge to entropy solutions of the pressureless Euler equation as $\alpha \to 0$.
\subsection{The nominal problem}  
\label{sec:nominal_problem}
In the special case of $\alpha = 0$, $F_{\alpha}^{f, \rmeas}$ reduces to 
\begin{equation}
    F_{0}^{f, \rmeas}\left(\diff\right) = 
    \begin{cases}
        \int \limits_{a}^{b} \frac{(\diff(x) - x)^2}{2} \D \rmeas(x) - f(\diff), & \text{if } \diff \text{ is monotone},\\
        \hfill \infty, \hfill & \text{else}.
    \end{cases}
\end{equation}
For functionals $F^{t u_0, \rmeas}_0$ where $f$ amounts to integration against $t u_0 \in L^2(\rmeas)$, minimizers of this functional are given by the $L^2$-projection $P_{\helly} \left(\mathrm{Id} + tu_0\right)$ of $ \mathrm{Id} + t u_0$ onto $\helly$.
IGR amounts to an entropic regularization of this projection operator. 
The entropic regularization of optimal transport and inequality-constrained problems is subject of much current attention \cite{cuturi2013sinkhorn,keith2023proximal,lindsey2023fast,li2023kernel}.
In the unbounded case ($a = -\infty, b = \infty$) the paths $t \mapsto P_{\helly} \left(\mathrm{Id} + tu_0\right)$ are known to describe the dynamics of sticky particles \cite{natile2009wasserstein,brenier2013sticky,cavalletti2015simple}.
In terms of the Lagrangian velocity $\dot{\diff}_t$, the Eulerian velocity density $u$ and mass distribution $\rho_t$ are given as in \cref{def:variational_solution}.
The resulting $u, \rho$ are weak solutions of the pressureless Euler equation (\cref{eqn:1d-igr-pressureless} with $\alpha = 0$) \cite{dermoune1999probabilistic,natile2009wasserstein,cavalletti2015simple,hynd2019lagrangian,hynd2020trajectory}.
The authors of \cite{suder2021lagrangian,carrillo2023equivalence} furthermore show that these weak solutions correspond to entropy solutions of the pressureless Euler equation in the sense of \cite{nguyen2008pressureless}. 
The optimality criterion / KKT condition of the projection $P_{\helly} (\mathrm{Id} + t u)$ corresponds to the entropy condition of the conservation law \cite{carrillo2023equivalence}.
We point out that the setting of finite $a, b$ reduces to a special case of the unbounded setting by mirroring the initial condition about $a$ and $b$. 
For instance, in the interval $[b, b + (b - a)],$ we can define the initial condition as $-u_0(2b - \cdot)$, and in the interval $[a - (b - a), a]$ as $- u_0(2a - \cdot)$. 
Repeating this construction ad infinitum yields a problem on $(-\infty, \infty)$ that satisfies $(\diff_{t}(a), \diff_{t}(b)) \equiv (a, b)$ by symmetry.   
We thus consider paths of the form $t \mapsto P_{\helly} (\mathrm{Id} + t u_0)$ and the resulting $u, \rho$ defined by \cref{eqn:eulerian_sticky_particles} as the natural notion of viscosity solution of the pressureless Euler equation. 
We study the convergence of solutions of \cref{eqn:1d-igr-pressureless} to these solutions as $\alpha \to 0$.
\subsection{\texorpdfstring{$\Gamma$}{Gamma}-convergence of \texorpdfstring{$F_{\alpha}^{f, \rmeas}$}{IGR} and convergence to entropy solutions}
\label{sec:gamma_convergence}
$\Gamma$-convergence provides a natural framework for studying the convergence of functionals, such as $\lim \limits_{\alpha \to 0} F_{\alpha}^{f, \rmeas}$.
\begin{definition}
    Let $\left(F_n\right)_{n \in \N}, F_n\colon X \to \R \cup \{\pm\infty\}$ be a sequence of functionals on a first countable space $X$. 
    We say that $F_{n}$ $\Gamma$-converges to $F_*$ if for all $x \in X$, 
    \begin{align}
        x_{n} \to x \implies \quad F_*(x) &\leq \liminf_{n \to \infty} F_n(x_n),\\
        \forall x \in X, \exists \left(x_{n}\right)_{n \in \N}: x_{n} \to x, \quad F_*(x) &\geq \limsup_{n \to \infty} F_n(x_n).  
    \end{align}
\end{definition}
The $x_n$ in the second part of the definition are called recovery sequences.
The following theorem establishes the convergence of minimizers of $\Gamma$-convergent functionals on compact spaces. 
\begin{theorem}
    \label{thm:fundamental_gamma_convergence}
    Let the topological space $X$ be compact and first countable, and let $F_n$ $\Gamma$-converge to $F_*$.
    Then, $F_*$ has a minimizer in $X$ and its minimum equals $\lim_{n \rightarrow \infty} \inf_{x \in X} F_n(x)$. 
    Furthermore, every accumulation point of a sequence $\left(x_{n}\right)_{n \in \N^*}$ satisfying $\lim_{n \rightarrow \infty} F_{n}(x_{n}) = \inf_{x \in X} F_{n}(x)$ is a minimizer of $F_*$.
\end{theorem}
\begin{proof}
    This follows from \cite[Theorem 2.10]{braides2006handbook}. Note that the requirement therein of $X$ being metric is only needed to ensure first countability.
\end{proof}
We now show that $F_{\alpha}^{f, \rmeas}$ $\Gamma$-converges to $F_{0}^{f, \rmeas}$.
\begin{theorem}
    Let $\ugen \in L^1(\rmeas), \alpha_k \rightarrow 0$ for $\rmeas \ll \mathscr{L}$, $0 <  \essinf \frac{\D \rmeas}{\D \mathscr{L}} \leq \esssup \frac{\D \rmeas}{\D \mathscr{L}} < \infty$.
    Then, $F_{\alpha_{k}}^{\ugen, \rmeas}$ $\Gamma$-converges to $F_{0}^{\ugen, \rmeas}$ with respect to pointwise convergence on $\helly$.
\end{theorem}
\begin{proof}
    The lim-inf inequality follows from the lower bound of $\Ent$ due to Jensen's inequality, by which $F_{0}^{f, \rmeas} \leq F_{\alpha}^{f, \rmeas} + C \alpha$ for all $\alpha \geq 0$ and a constant $C$, together with the lower semicontinuity of the squared norm and $\diff \mapsto \int_{a}^{b} \diff \cdot \ugen \D \rmeas$. 
    The latter two properties follow from the dominated convergence theorem and uniform boundedness of elements of $\helly$.
    To prove the lim-sup inequality, we now construct a recovery sequence for an arbitrary $\diff \in \helly$.
    We construct a sequence of monotone $\left(\diff_{k}\right)_{k \in \N^*} \in \helly$ that converges to $\diff$, with each element having finite entropy.
    Consider the sequence defined by $\diff_{k} = \left(1 - \frac{1}{k}\right) \diff + \frac{1}{k} \mathrm{Id}$.
    Then, $\diff_{k} \in \helly$ and $\diff_{k} \to \diff$ pointwise.
    The gradient of the almost everywhere differentiable function $\diff_{k}$ is lower bounded by $1/k$. 
    Thus, by \cite[Theorem 11.1]{villani2009optimal}, $\frac{\D \left(\diff_{k}\right)_{\#}\rmeas}{\D \rmeas} \leq \frac{k \esssup \frac{\D\rmeas }{\D \Leb}}{\essinf \frac{\D \rmeas}{\D \Leb}} < \infty$, implying finiteness of $\Ent(\left(\diff_{k}\right)_{\#}\rmeas)$.
    Define now the sequence $\left(l_k\right)_{k \in \N^*}$ by 
    \begin{equation}
        l_k = \max \left(\left\{l \le k \middle| \sqrt{\alpha_k} \Ent\left(\left(\diff_{l}\right)_{\#}\rmeas \right) \leq 1 \right\} \cup \{1\}\right).
    \end{equation}
    By finiteness of each $\Ent\left(\left(\diff_{l}\right)_{\#}\rmeas \right)$ and $\sqrt{\alpha_k} \rightarrow 0$, the sequence $l_k$ diverges to infinity and $\diff_{l_k} \to \diff$ pointwise. 
    For large $k$, $\alpha_k \Ent\left(\left(\diff_{l_k}\right)_{\#}\rmeas \right) \leq \sqrt{\alpha_k} \to 0$ by choice of $l_k$. 
    Thus, $\left(\diff_{l_k}\right)_{k \in \N^*}$ is a recovery sequence for $\diff$.
\end{proof}
As indicated in \cref{eqn:variational_exp}, the functional of interest is the minimizer of $F_{\alpha}^{f^{t \uvel_0}_{\alpha}, \rmeas}$, for 
\begin{equation}
  f_{\alpha}^{u_0}(\diff) \coloneqq \int \limits_{a}^{b} \diff \cdot u_0 + \alpha \partial_x \diff \cdot \left(\partial_x u_0 \right) \D \rmeas(x) 
  - \int \limits_{a}^{b} \alpha \partial_x \diff \D \rmeas(x),
\end{equation}
motivating the need to show convergence of this functional to $F_{0}^{tu_0, \rmeas}$, as well.
\begin{corollary}
    Let $\rmeas \ll \Leb$ with $\frac{\D \rmeas}{\D \Leb} \in W^{1,1}$, $0 <  \essinf \frac{\D \rmeas}{\D \mathscr{L}} \leq \esssup \frac{\D \rmeas}{\D \mathscr{L}} < \infty$ and $u_0 \in W^{2,1}$.
    Then, for any $t \geq 0$ and positive sequence $\alpha_k \rightarrow 0$, $F_{\alpha_{k}}^{f^{tu_0}_{\alpha_{k}}, \rmeas}$ $\Gamma$-converges to $F_{0}^{f^{tu_0}_{0}, \rmeas}$.
\end{corollary}
\begin{proof}
    By the regularity assumptions on $\rmeas$ and $u_0$, the functionals $f_{\alpha_{k}}^{t u_{0}}$ are represented by integration against a uniformly bounded sequence of functions $u_{\alpha_k}$ in $L^1(\rmeas)$.
    By absolute continuity of functions in $W^{1,1}$ and the regularity assumptions on $\rmeas$ and $\uvel_0$, the sequence $\left(\alpha_{k}^{-1}\left(u_{\alpha_{k}} - t u_{0}\right)\right)_{k \in \N^*}$ is uniformly bounded in $L^1(\rmeas)$.
    Thus, the functional $\left \langle \left(\alpha_{k}^{-1}\left(u_{\alpha_{k}} - t u_{0}\right)\right), \cdot \right \rangle_{L^2(\rmeas)}$ is uniformly bounded over $\helly$ and for all convergent sequences $\left(\diff_{n}\right) \in \helly$ we have 
    \begin{equation}
    \liminf \limits_{n \rightarrow \infty} F_{\alpha_{n}}^{f^{t u_0}_{\alpha_{n}}, \rmeas}\left(\diff_n\right) = \liminf \limits_{n \rightarrow \infty} F_{\alpha_{n}}^{t u_{0}, \rmeas}\left(\diff_n\right),  \quad \text{and} \quad \limsup \limits_{n \rightarrow \infty} F_{\alpha_{n}}^{f^{t u_0}_{\alpha_{n}}, \rmeas}\left(\diff_n\right) = \limsup \limits_{n \rightarrow \infty} F_{\alpha_{n}}^{t u_{0}, \rmeas}\left(\diff_n\right),
    \end{equation}
    from which the result follows, by the previous theorem.
\end{proof}
We conclude the convergence of minimizers of $F_{\alpha}^{f^{t \uvel_0}_{\alpha}, \rmeas}$ to the minimizer of $F_{0}^{tu_0, \rmeas}$.
\begin{corollary}
    Let $\rmeas \ll \Leb$ with $\frac{\D \rmeas}{\D \Leb} \in W^{1,1}, 0 < \essinf \frac{\D \rmeas}{\D \mathscr{L}} \leq \esssup \frac{\D \rmeas}{\D \mathscr{L}} < \infty$ and $u_0 \in W^{2,1}$.
    Any sequence of minimizers $\diff_{t, \alpha}$ of the functional $F_{\alpha}^{f^{t \uvel_0}_{\alpha}, \rmeas}$ converges pointwise to the unique minimizer of $F_{0}^{t u_0, \rmeas}$. 
\end{corollary}
\begin{proof}
    The result follows from the previous corollary and \cref{thm:fundamental_gamma_convergence} by using the fact that if any subsequence of a sequence in a topological space $X$ has a subsequence converging to a point $x^* \in X$, then the entire sequence converges to $x^*$. 
\end{proof}

\section{Spatial regularity of minimizers}
\label{sec:weak_solutions}
\subsection*{Overview}
This section establishes the spatial regularity of minimizers of $F^{\ugen, \rmeas}_{\alpha}$.
By constructing a suitable variation, \cref{sec:rho_lower_bound} proves a lower bound on the pushforward density of minimizers of $F_{\alpha}^{v, \rmeas}$.
\Cref{sec:absolute_continuity} uses this result to show that the minimizers are absolutely continuous, with uniformly bounded weak derivatives.
Using another variation, \cref{sec:lower_bound_derivative} proves a lower bound for the minimizer's derivative.
\Cref{sec:higher_order_regularity} uses a bootstrap argument to prove higher regularity of minimizers that matches the regularity of the problem data.
\subsection{Lower bound on the pushforward density}
\label{sec:rho_lower_bound}
The key step in proving regularity of the minimizers of $F_{\alpha}^{\ugen, \rmeas}$ is to establish a lower bound on the pushforward density $\diff^*_{\#} \rmeas$ of the minimizer $\diff^*$.
This is the most delicate step of our argument, since the function $x \mapsto x \log(x)$ in the definition of $\Ent\left(\cdot \right)$ is bounded in the limit of small $x$.
Likewise, the function $x \mapsto -\log(x)$ in $\hat{F}_{\alpha}$ grows very slowly, amounting to very weak control on the size of $\partial_x \diff^*$.
By the change of variables formula, lower bounds on pushforward densities imply upper bounds on $\partial_x \diff^*$, enabling us to make first progress on establishing regularity of $\diff^*$.
\begin{lemma}\label[lemma]{le:rho_lower_bound}
  For a probability measure $\rmeas$ on $[a, b]$ with $\rmeas \ll \Leb$, write $\mu = \frac{\D \rmeas}{\D \Leb}$ and suppose that:
  \begin{enumerate}
    \item[(1)] $\mu \in W^{1, \infty}([a, b])$.
    \item[(2)] $0 <  C_{\mathrm{min}}(\rmeas) \le \essinf \mu(\cdot)$ and $\esssup \mu(\cdot) \le C_{\mathrm{max}}(\rmeas) < \infty$.
  \end{enumerate}
  Let $\diff^*$ be the minimizer of $F_{\alpha}^{\ugen, \rmeas}$ (its uniqueness follows from \cref{thm:uniqueness_minimizers}) for $\alpha > 0$ and $\ugen \in L^\infty$.
  Set $\D x \coloneqq \D \Leb$ and define
    $\rho \defeq \frac{\D \diff^*_{\#} \rmeas}{\D x}$.
  Furthermore, suppose that $\norm{\ugen}_{L^{\infty}} \le C_{\max}^{\mathrm{vel}}$.
  Then, there exists a constant $C(C_{\max}^{\textup{vel}}, \rmeas, \norm{\partial_x\mu}_{L^\infty}, \alpha, a, b) > 0$
  such that
  \begin{equation}
    \essinf_{y \in [a, b]} \rho(y) \ge C(C_{\max}^{\textup{vel}}, \rmeas, \norm{\partial_x\mu}_{L^\infty}, \alpha, a, b).
  \end{equation}
\end{lemma}
\begin{proof}
To simplify notation, we define $\kappa \coloneqq 2\frac{C_{\max}(\rmeas)}{C_{\min}(\rmeas)}$
and constants
\begin{align}
  C_R &\defeq \frac{\norm{\partial_x \rmeas}_{L^{\infty}}}{C_{\min}(\rmeas)} \left( 2 \frac{C_{\max} (\rmeas)}{C_{\min}(\rmeas)} - 1 \right),\\
  E_{\rho} &\defeq
  \frac{2\max\{\abs{a}, \abs{b}\} C_{\max}^{\mathrm{vel}}}{\alpha}
  + \abs{\Ent(\rmeas)}
  + \frac{2(b-a)}{e},\\
  C_{\clubsuit_2} &\defeq
  \alpha C_R\left(E_{\rho} + 2\frac{C_{\max}(\rmeas)}{C_{\min}(\rmeas)}\right),\\
  \tilde{C}
  &\defeq
  2 \kappa C_{\max}(\rmeas) (b - a)^2 +  2 \kappa C_{\max}(\rmeas) (b - a) \norm{\ugen}_{L^{\infty}}
  +  \kappa \frac{2 \alpha}{e}.
\end{align}
Toward a contradiction, suppose that $\essinf_{y \in [a, b]} \rho(y) = 0$.
\begin{figure}
  \begin{tikzpicture}
\begin{groupplot}[
	compat=1.3,
	group style={group size=3 by 1,
	horizontal sep=0.25cm,
    vertical sep=0.5cm,},
	]
	\nextgroupplot[
		standard,
 		xlabel={$x$},
		ylabel={$\diff(x)$},
		height=0.15\textwidth,
		width=0.30\textwidth,
        ymin=-0.10,
        ymax=1.1,
        xmin=-1.10,
        xmax=1.10,
        yticklabels={{$a$,$\hat{y}$, $b$}},
        xticklabels={{$a$,$\hat{x}$, $b$}},
		xtick={-1, 0.15, 1},
		ytick={0, 0.5, 1},
		enlarge x limits=0.,
		enlarge y limits=0.,
		legend style={
			at={(1.480,1.00)},inner sep=3pt,anchor=south,legend columns=5,legend cell align={left}, draw=none,fill=none, /tikz/every even column/.append style={column sep=0.10cm}},
		]

		\addlegendimage{silver,ultra thick}
		\addlegendimage{joshua, ultra thick, dashed}
		\addlegendimage{joshua, ultra thick}
		\addlegendimage{rust, ultra thick, dashed}
		\addlegendimage{orange, ultra thick}

		\legend{{$\diff^*(\cdot)\!=\!\diff^*(\cdot, 0)$},
				{regularize $[\hat{x}\! - \!\eps, \hat{x}]$},
				{$\diff^*(\cdot, \eps), \delta = \eps$},
				{$[\hat{x}\!-\! \eps, \hat{x}] {\color{rust} \to}[\hat{x}\!-\!\delta,\hat{x}]$},
				{$\diff^*(\cdot, \eps), \delta = 2 \epsilon$}}

		\addplot[silver,ultra thick,domain=-1:0.15,samples=100] {(x + 1) / 3 + sin(deg(2 * pi * (x + 1))) / 20};
		\addplot[silver,ultra thick,domain=0.15:1,samples=100] {(x + 2) / 3 + sin(deg(2 * 2 * pi * x)) / 35};
		\addplot[joshua,ultra thick,dashed,domain=-0.075:0.15,samples=100] {((-0.075 + 1) / 3 + sin(deg(2 * pi * (-0.075 + 1))) / 20) + (x + 0.075) / 0.225 * (((0.15 + 2) / 3 + sin(deg(2 * 2 * pi * 0.15)) / 35) - ((-0.075 + 1) / 3 + sin(deg(2 * pi * (-0.075 + 1))) / 20))};

	\nextgroupplot[
		standard,
 		xlabel={$x$},
		ylabel={},
		height=0.15\textwidth,
		width=0.30\textwidth,
        ymin=-0.10,
        ymax=1.1,
        xmin=-1.10,
        xmax=1.10,
        yticklabels={{$a$,$\hat{y}$, $b$}},
        xticklabels={{$a$, $\hat{x} - \eps$, $\hat{x}$, $b$}},
		xtick={-1, -0.075, 0.15, 1},
		ytick={0, 0.5, 1},
		ymajorticks=false,
		enlarge x limits=0.,
		enlarge y limits=0.,
		legend style={
			at={(1.80,1.00)},inner sep=3pt,anchor=south,legend columns=6,legend cell align={left}, draw=none,fill=none, /tikz/every even column/.append style={column sep=0.20cm}},
		]

		\addplot[silver,ultra thick,domain=-1:0.15,samples=100] {(x + 1) / 3 + sin(deg(2 * pi * (x + 1))) / 20};
		\addplot[silver,ultra thick,domain=0.15:1,samples=100] {(x + 2) / 3 + sin(deg(2 * 2 * pi * x)) / 35};

		\addplot[joshua,ultra thick,domain=-1:-0.075,samples=100] {(x + 1) / 3 + sin(deg(2 * pi * (x + 1))) / 20};
		\addplot[joshua,ultra thick,domain=-0.075:0.15,samples=100] {((-0.075 + 1) / 3 + sin(deg(2 * pi * (-0.075 + 1))) / 20) + (x + 0.075) / 0.225 * (((0.15 + 2) / 3 + sin(deg(2 * 2 * pi * 0.15)) / 35) - ((-0.075 + 1) / 3 + sin(deg(2 * pi * (-0.075 + 1))) / 20))};
		\addplot[joshua,ultra thick,domain=0.15:1,samples=100] {(x + 2) / 3 + sin(deg(2 * 2 * pi * x)) / 35};

		\addplot[rust, ultra thick, domain=-0.30:-0.075,samples=100, <-, dashed] {0.5};
		\draw[ultra thick, rust, dashed] (axis cs: -0.075,-0.1) -- (axis cs: -0.075,2.0);

	\nextgroupplot[
		standard,
 		xlabel={$x$},
		ylabel={},
		height=0.15\textwidth,
		width=0.30\textwidth,
        ymin=-0.10,
        ymax=1.1,
        xmin=-1.10,
        xmax=1.10,
        yticklabels={{$a$,$\hat{y}$, $b$}},
        xticklabels={{$a$, $\hat{x} - \delta$, $\hat{x}$, $b$}},
		xtick={-1, -0.3, 0.15, 1},
		ytick={0, 0.5, 1},
		ymajorticks=false,
		enlarge x limits=0.,
		enlarge y limits=0.,
		legend style={
			at={(1.80,1.00)},inner sep=3pt,anchor=south,legend columns=6,legend cell align={left}, draw=none,fill=none, /tikz/every even column/.append style={column sep=0.20cm}},
		]

		\addplot[orange,ultra thick,domain=-1:-0.3,samples=100] {0.925 / 0.7 * (x + 1) / 3 + sin(deg(2 * pi * 0.925 / 0.7 * (x + 1))) / 20};
		\addplot[orange,ultra thick,domain=-0.3:0.15,samples=100] {((-0.075 + 1) / 3 + sin(deg(2 * pi * (-0.075 + 1))) / 20) + (x + 0.3) / 0.45 * (((0.15 + 2) / 3 + sin(deg(2 * 2 * pi * 0.15)) / 35) - ((-0.075 + 1) / 3 + sin(deg(2 * pi * (-0.075 + 1))) / 20))};
		\addplot[orange,ultra thick,domain=0.15:1,samples=100] {(x + 2) / 3 + sin(deg(2 * 2 * pi * x)) / 35};

		\draw[ultra thick, rust, dashed] (axis cs: -0.3,-0.1) -- (axis cs: -0.3,2.0);

\end{groupplot}

\end{tikzpicture}
  \caption{\textbf{The variation $\diff(\cdot, \eps)$} in \eqref{eqn:variation} consists of two steps. For $\delta = \eps > 0$, it simply replaces $\diff$ on the interval $[\hat{x} - \eps, \hat{x}]$ by a linear interpolation. For $\delta > \eps$ it stretches the resulting function values on $[\hat{x} - \eps, \hat{x}]$ to $[\hat{x} - \delta, \hat{x}]$ and shrinks the remaining region $[a, \hat{x} - \delta]$ of the input space.
  The first step replaces parts of the output of $\diff$ while the second only rearranges outputs.}
  \label{fig:variation}
\end{figure}
Then, for all $\Lambda > 0$, the set $\left\{ \rho \le \Lambda \right\}$ has positive Lebesgue measure.
By \cite[Theorem 2]{jessen1935note}, for arbitrarily small $\Lambda > 0$, there exists $\hat{y} \in [a, b]$ such that
for any sequence of intervals $I \to \left\{ \hat{y} \right\}$  containing $\hat{y}$ with $\lim \Leb(I) = 0$,
\begin{equation}
  \lim_{I \to \{\hat{y}\}} \frac{1}{\abs{I}} \int \limits_{I} \rho(y)  \D y
  = \rho(\hat{y}) = \Lambda.
\end{equation}
By choosing $1 > \Lambda > 0$ sufficiently small we can ensure that
\begin{align}
\tilde{C}
+ C_{\clubsuit_2}
+ \alpha \left(C_{\max}\left(\rmeas\right) \log \Lambda
- \kappa C_{\max}(\rmeas) \log \left( \frac{C_{\min}(\rmeas)}{\kappa C_{\max}(\rmeas)} \right) \right) \le -1.
\end{align}
Choose $\hat{x} \in [a, b]$ such that $\sup \left\{ x: \diff^*(x) \le \hat{y} \right\} \le \hat{x} \le \inf \left\{ x: \diff^*(x) \ge \hat{y} \right\}$.
Let $\eps > 0$.
Without loss of generality, $\hat{x} \in [(a + b)/2, b]$, the alternative case is treated symmetrically.
As illustrated in \cref{fig:variation}, we define $\diff(x, \eps)$ as:
  \begin{equation}
    \label{eqn:variation}
    \diff(x, \eps)
    =
    \begin{cases}
      \diff^*\left( \frac{\hat{x} - \eps - a}{\hat{x} - \delta - a} (x - a) + a \right) & x \in [a, \hat{x} - \delta) \\
      \frac{\diff^*(\hat{x}) - \diff^*(\hat{x} - \eps)}{\delta} (x - (\hat{x} - \delta)) + \diff^*(\hat{x} - \eps) & x \in [\hat{x} - \delta, \hat{x}) \\
      \diff^* \left( x \right) & x \in [\hat{x}, b]
    \end{cases},
  \end{equation}
where we set $\delta = 2\frac{C_{\max}\left(\rmeas\right)}{C_{\min}\left(\rmeas\right)} \eps = \kappa \eps$ but maintain distinct names to improve transparency of their influence.
We write
\begin{equation}
  R_{\eps}(y)
  =
  \frac{
    \mu\left(\frac{\hat{x} - \delta - a}{\hat{x} - \eps - a}
    \left((\diff^*)^{-1}(y)-a\right) + a\right)
  }{
    \mu\left((\diff^*)^{-1}(y)\right)
  }.
\end{equation}
Then, via the change of the variable formula, the density of $\diff(\cdot, \eps)_{\#} \rmeas$ (w.r.t. $\Leb$) is given by
\begin{equation}
  \rho(y, \eps)
  =
  \begin{cases}
    \frac{\hat{x} - \delta - a}{\hat{x} - \eps - a} R_{\eps}(y) \rho(y)
    & y \in [a, \diff^*(\hat{x} - \eps)) \\
    \frac{ \delta }{\diff^*(\hat{x}) - \diff^*(\hat{x} - \eps)}
    \mu \left( \frac{\delta (y - \diff^*(\hat{x}- \eps))}{\diff^*(\hat{x}) - \diff^*(\hat{x} - \eps)} + \hat{x} - \delta \right)
      &
    y \in [\diff^*(\hat{x} - \eps), \diff^*(\hat{x})) \\
    \rho(y) & y \in [\diff^*(\hat{x}), b]
  \end{cases}.
\end{equation}
We now show that for all $\eps > 0$ small enough,
$F^{\ugen, \rmeas}_{\alpha}(\diff(\cdot, \eps)) - F^{\ugen, \rmeas}_{\alpha}(\diff^*(\cdot))< 0$.
By the definition of $\diff(\cdot, \eps)$,
\begin{align}
  &\frac{1}{\eps} \cdot \left[F^{\ugen, \rmeas}_{\alpha}(\diff(\cdot, \eps)) - F^{\ugen, \rmeas}_{\alpha}(\diff^*(\cdot)) \right] \\
  &=
  \frac{1}{\eps}
   \left[\int \limits_a^b \tfrac{1}{2}\left(\diff(x, \eps) - x\right)^2
  - \diff(x, \eps) \cdot \ugen(x) \D \rmeas(x)
+ \alpha \int \limits_a^b \rho(y, \eps) \log \rho(y, \eps)  \D y  \right] \\
  &\quad - \frac{1}{\eps}
   \left[\int \limits_a^b \tfrac{1}{2}\left(\diff^*(x) - x\right)^2
  - \diff^*(x) \cdot \ugen(x) \D \rmeas(x)
+ \alpha \int \limits_a^b \rho(y) \log \rho(y) \D y \right] \\
  &=
  \underbrace{ \frac{1}{\eps}
    \left[
    \int \limits_a^b \tfrac{1}{2}\left(\diff(x, \eps) - x\right)^2
  - \diff(x, \eps) \cdot \ugen(x)
-   \tfrac{1}{2}\left(\diff^*(x) - x\right)^2
  + \diff^*(x) \cdot \ugen(x) \D \rmeas(x)
\right]}_{\heartsuit} \\
  &\quad +
  \underbrace{\frac{1}{\eps}
    \left[ \alpha
      \int \limits_a^{\diff^*(\hat{x} - \eps)} \left( \frac{\hat{x} - \delta - a}{\hat{x} - \eps - a}  R_{\eps}(y) - 1 \right) \rho(y) \log \rho(y)
      +
  \frac{\hat{x} - \delta - a}{\hat{x} - \eps - a} R_{\eps}(y) \rho(y) \log \left( \frac{\hat{x} - \delta - a}{\hat{x} - \eps - a} R_{\eps}(y) \right) \D y
\right]}_{\clubsuit}\\
    & \quad +
    \underbrace{\frac{1}{\eps}
      \begin{bmatrix}
        \alpha
        \int \limits_{\diff^*(\hat{x} - \eps)}^{\diff^*(\hat{x})}
      \frac{ \delta }{\diff^*(\hat{x}) - \diff^*(\hat{x} - \eps)}
      \mu \left( \frac{\delta (y - \diff^*(\hat{x}- \eps))}{\diff^*(\hat{x}) - \diff^*(\hat{x} - \eps)} + \hat{x} - \delta \right) \qquad \qquad \qquad \qquad \qquad \qquad \qquad \qquad\\
   \times \log \left( \frac{\delta}{\diff^*(\hat{x})
       - \diff^*(\hat{x} - \eps)} \mu \left( \frac{\delta (y - \diff^*(\hat{x}- \eps))}{\diff^*(\hat{x}) - \diff^*(\hat{x} - \eps)} + \hat{x} - \delta \right) \right)
       - \rho(y) \log \rho(y) \D y
      \end{bmatrix}
 }_{\diamondsuit}
\end{align}
We will proceed as follows:
\begin{enumerate}
  \item[(I)] First, we show that $\heartsuit$ is upper bounded for all sufficiently small $\eps$.
  \item[(II)] Next, we show that $\clubsuit$ is upper bounded for all sufficiently small $\eps$.
  \item[(III)] Finally, we show that as $\eps \rightarrow 0$, $\diamondsuit$ tends to $-\infty$, yielding $F^{\ugen, \rmeas}_{\alpha}(\diff(\cdot, \eps)) - F^{\ugen, \rmeas}_{\alpha}(\diff^*(\cdot))< 0$.
\end{enumerate}
\textbf{(Step 1)}
In this step, we claim that $\heartsuit$ is upper bounded.
First, note that
\begin{align}
  \heartsuit
  =\phantom{+}
  \frac{1}{\eps}
  \int \limits_a^b
  \frac{1}{2}\left(\diff(x, \eps) - x\right)^2
  - \frac{1}{2} \left(\diff^*(x) - x\right)^2 \D \rmeas(x)
  +\frac{1}{\eps}
  \int \limits_a^b - \diff(x, \eps) \cdot \ugen(x)
  + \diff^*(x) \cdot  \ugen(x) \D \rmeas(x).
\end{align}
First, expanding all the terms using the definition of $\diff(x, \eps)$, we have
\begin{align}
  &
  \int \limits_a^{\hat{x}} \frac{1}{2}\left(\diff(x, \eps) - x\right)^2 \D \rmeas(x)
  = \quad \int \limits_a^{\hat{x} - \delta}
  \frac{1}{2} \left(\diff(x, \eps) - x\right)^2 \mu(x) \D x
  + \int \limits_{\hat{x} - \delta}^{\hat{x}}
  \frac{1}{2} \left(\diff(x, \eps) - x\right)^2 \mu(x) \D x \\
  &=
  \int \limits_{a}^{\hat{x} - \delta} \frac{1}{2} \left(\diff^* \left( \frac{\hat{x} - \eps - a}{\hat{x} - \delta - a} (x - a) + a \right) - x \right)^2 \mu(x) \D x \\
  & \quad + \int \limits_{\hat{x} - \delta}^{\hat{x}}
  \frac{1}{2} \left( \frac{\diff^*(\hat{x}) - \diff^*(\hat{x} - \eps)}{\delta}
  \left( x - (\hat{x} - \delta) \right) + \diff^*(\hat{x} - \eps) - x\right)^2 \mu(x) \D x.
\end{align}
Then, we compute
\begin{align}
  &\frac{1}{\eps}
  \int \limits_{a}^{\hat{x} - \delta} \frac{1}{2}
  \left(\diff^* \left( \frac{\hat{x} - \eps - a}{\hat{x} - \delta - a} (x - a) + a \right) - x \right)^2  \D\rmeas(x)
  - \frac{1}{\eps} \int \limits_a^{\hat{x}} \frac{1}{2} \left( \diff^*(x) - x \right)^2 \D \rmeas(x) \\
  &= \frac{1}{\eps}
  \int \limits_a^{\hat{x} - \delta}
  \frac{1}{2} \left( \diff^* \left( \frac{\hat{x} - \eps - a}{\hat{x} - \delta - a} (x - a) + a \right) - \diff^*(x) \right)
  \cdot \left( \diff^* \left( \frac{\hat{x} - \eps - a}{\hat{x} - \delta - a} (x - a) + a \right) + \diff^*(x) - 2 x  \right) \D \rmeas(x) \\
  &\quad + \frac{1}{2\eps} \int \limits_{\hat{x} - \delta}^{\hat{x}}
  \left( \diff^*(x) - x \right)^2 \D \rmeas(x) \\
  &\le (b -a) C_{\max}(\rmeas) \int \limits_a^{\hat{x} - \delta}
  \frac{1}{\eps} \left( \diff^* \left( \frac{\hat{x} - \eps - a}{\hat{x} - \delta - a} (x - a) + a \right) - \diff^*(x) \right) \D x + \frac{\kappa}{2} (b - a)^2 C_{\text{max}}(\rmeas).
\end{align}
Therefore,
\begin{align}
  &\limsup_{\eps \to 0}
\frac{1}{\eps}
  \int \limits_{a}^{\hat{x} - \delta} \frac{1}{2}
  \left(\diff^* \left( \frac{\hat{x} - \eps - a}{\hat{x} - \delta - a} (x - a) + a \right) - x \right)^2  \D\rmeas(x)
  - \frac{1}{\eps} \int \limits_a^{\hat{x}} \frac{1}{2} \left( \diff^*(x) - x \right)^2 \D \rmeas(x) \\
  &\le \tfrac{3}{2} \kappa C_{\max}(\rmeas) (b-a)^2.
\end{align}
Next, to obtain the limit of the remaining term, we compute
\begin{align}
  & \limsup_{\eps \to 0} \frac{1}{\eps}
  \int\limits_{\hat{x} - \delta}^{\hat{x}}
  \frac{1}{2} \left(\frac{\diff^*(\hat{x}) - \diff^*(\hat{x} - \eps)}{\delta} (x - (\hat{x} - \delta)) + \diff^*(\hat{x} - \eps) - x\right)^2 \mu(x) \D x \\
  & \le C_{\max}(\rmeas) \limsup_{\eps \to 0} \frac{1}{2\eps}
  \int\limits_{\hat{x} - \delta}^{\hat{x}}
  \left(\frac{\diff^*(\hat{x}) - \diff^*(\hat{x} - \eps)}{\delta} (x - (\hat{x} - \delta)) \right)^2 \\
  &\qquad\qquad\qquad\qquad\qquad \quad + 2 \left(\frac{\diff^*(\hat{x}) - \diff^*(\hat{x} - \eps)}{\delta} (x - (\hat{x} - \delta)) \right)(b - a)
  +(b - a)^2 \D x \\
  &\le \frac{C_{\max}(\rmeas)}{2} \limsup_{\eps \to 0}
   \frac{\delta}{\eps}( b - a )^2
  \le \frac{1}{2}\kappa C_{\max}(\rmeas)  \left( b - a \right)^2.
\end{align}
Overestimating slightly allows us to simplify the above to
\begin{equation} 
  \limsup_{\eps \to 0} \frac{1}{\eps} \int \limits_a^b \frac{1}{2} \left(\diff(x, \eps) - x\right)^2
  - \frac{1}{2} \left(\diff^*(x) - x\right)^2 \D \rmeas(x)
  \le 2 \kappa C_{\max}(\rmeas) (b - a)^2.
\end{equation}
Now, we estimate the term
\begin{align}
  &\frac{1}{\eps}
  \int \limits_a^b  \left( \diff^*(x) - \diff(x, \eps) \right) \cdot \ugen(x) \D \rmeas(x)
  =
  \frac{1}{\eps}
  \int \limits_a^{\hat{x} - \delta}  \left( \diff^*(x) - \diff(x, \eps) \right) \cdot \ugen(x) \mu(x)  \D x \\
  &\qquad \qquad \qquad \qquad \qquad
  \qquad \qquad \qquad \qquad \qquad
  + \frac{1}{\eps}\int \limits_{\hat{x} - \delta}^{\hat{x}}
  \left(\diff^*(x) - \diff(x, \eps) \right) \cdot \ugen(x) \mu(x) \D x.
\end{align}
For the first term, we compute
  \begin{align}
    & \frac{1}{\eps} \int \limits_{a}^{\hat{x} - \delta}  \left(\diff^*\left(\frac{\hat{x} - \eps - a}{\hat{x} - \delta - a} (x - a) + a\right) - \diff^*\left(x \right) \right) \cdot \ugen(x) \mu(x) \D x  \\
    &\le \norm{\ugen}_{L^{\infty}} C_{\max}(\rmeas)
     \int \limits_a^{\hat{x} - \delta}
     \frac{1}{\eps}
\left(\diff^*\left(\frac{\hat{x} - \eps - a}{\hat{x} - \delta - a} (x - a) + a\right) - \diff^*\left(x \right) \right)  \D x.
  \end{align}
By applying a change of variables to the $\diff^*(x)$ parts of the integrals, we can conclude
\begin{align}
    & \limsup_{\eps \to 0} \frac{1}{\eps}
      \int \limits_{a}^{\hat{x} - \delta}  \left(\diff^*\left(\frac{\hat{x} - \eps - a}{\hat{x} - \delta - a} (x - a) + a\right) - \diff^*\left(x \right) \right) \cdot \ugen(x) \mu(x)\D x  \\
    \qquad
\\
    &\le (\kappa - 1) C_{\max}(\rmeas)  (b - a) \norm{\ugen}_{L^{\infty}}
    < \kappa C_{\max}(\rmeas)  (b - a) \norm{\ugen}_{L^{\infty}}.
\end{align}
For the last term,  we have the following estimates
\begin{align}
  &\limsup_{\eps \to 0}
  \frac{1}{\eps}  \int \limits_{\hat{x} - \delta}^{\hat{x}}
  \left(\diff^*(x) - \frac{\diff^*(\hat{x}) - \diff^*(\hat{x} - \eps)}{\delta}
  \left( x - (\hat{x} - \delta) \right) - \diff^*(\hat{x} - \eps) \right)
  \cdot \ugen(x) \D x \\
  &\le
  \limsup_{\eps \to 0}
\frac{C_{\max}(\rmeas) \norm{\ugen}_{L^{\infty}}}{\eps}  \int \limits_{\hat{x} - \delta}^{\hat{x}}
  \abs{\diff^*(x) - \frac{\diff^*(\hat{x}) - \diff^*(\hat{x} - \eps)}{\delta}
  \left( x - (\hat{x} - \delta) \right) - \diff^*(\hat{x} - \eps) }
   \D x \\
  &\le
  \limsup_{\eps \to 0}
  C_{\max}(\rmeas) \norm{\ugen}_{L^{\infty}}
  \left( \frac{1}{\eps} \int\limits_{\hat{x} - \delta}^{\hat{x}}  \frac{\diff^*(\hat{x}) - \diff^*(\hat{x} - \eps)}{\delta}
  \left( x - (\hat{x} - \delta) \right) \D x + \kappa (b - a)\right) \\
  &\le \kappa C_{\max}(\rmeas) \norm{\ugen}_{L^{\infty}} (b - a).
\end{align}
Combining all the estimates above, we have
\begin{align}
  \limsup_{\eps \to 0} \heartsuit
  &\le 2 \kappa C_{\max}(\rmeas) (b-a)^2 + 2 \kappa C_{\max}(\rmeas) (b-a)\norm{\ugen}_{L^\infty}.
\end{align}
\textbf{(Step 2)}
In this step, we show that $\clubsuit$ is upper bounded.
We split $\clubsuit$ into the following two terms:
\begin{align}
  \clubsuit
  &= \underbrace{
  \frac{\alpha}{\eps}
    \int \limits_a^{\diff^*(\hat{x} - \eps)}
      \left( \frac{\hat{x} - \delta - a}{\hat{x} - \eps - a} - 1\right) \rho(y) \log \rho(y)
      +
      \left(\frac{\hat{x} - \delta - a}{\hat{x} - \eps - a} \right) \rho(y) \log \left(\frac{\hat{x} - \delta - a}{\hat{x} - \eps - a} \right) \D y
  }_{\clubsuit_{1}} \\
  &\quad +
  \underbrace{
  \frac{\alpha}{\eps}
\left[
      \begin{multlined}
    \int \limits_a^{\diff^*(\hat{x} - \eps)}
      \frac{\hat{x} - \delta - a}{\hat{x} - \eps - a}\left(R_{\eps}(y) - 1\right) \rho(y) \log \rho(y) \\
      \quad + \frac{\hat{x} - \delta - a}{\hat{x} - \eps - a} \left(R_{\eps}(y) - 1\right) \rho(y) \log \left( \frac{\hat{x} - \delta - a}{\hat{x} - \eps - a} \right) 
      + \frac{\hat{x} - \delta - a}{\hat{x} - \eps - a} R_{\eps}(y)\rho(y)\log R_{\eps}(y) \D y
      \end{multlined}
    \right]}_{\clubsuit_2}
\end{align}
Recalling that $\delta = 2 \frac{C_{\max}\left(\rmeas\right)}{C_{\min}\left(\rmeas\right)}\eps$ and, as $z \to 0$, $\log (1 - z) = -z - O(z^2)$, we upper bound $\clubsuit_1$ as follows:
\begin{align}
  &\frac{1}{\eps}
    \left(
      \int \limits_a^{\diff^*(\hat{x} - \eps)} \left( \frac{\hat{x} - \delta - a}{\hat{x} - \eps - a} - 1 \right) \rho(y) \log \rho(y)
      +
  \left(\frac{\hat{x} - \delta - a}{\hat{x} - \eps - a} \right) \rho(y) 
  \log \left(\frac{\hat{x} - \delta - a}{\hat{x} - \eps - a} \right) \D y
\right)\\
  &=
  \int \limits_a^{\diff^*(\hat{x} - \eps)}
  \left( \frac{1 - 2\frac{C_{\max}(\rmeas)}{C_{\min}(\rmeas)}}{\hat{x} - a - \eps} \right) \rho(y)  \log \rho(y)
  + \left(\frac{\hat{x} - \delta - a}{\hat{x} - \eps - a} \right) \frac{\rho(y)}{\eps} \log \left( 1 + \frac{\left(1 - 2\frac{C_{\max}(\rmeas)}{C_{\min}(\rmeas)}\right)\eps}{\hat{x} - a - \eps} \right) \D y\\
  &=
  \int \limits_a^{\diff^*(\hat{x} - \eps)}
  \left( \frac{1 - 2 \frac{C_{\max}(\rmeas)}{C_{\min}(\rmeas)}}{\hat{x} - a - \eps} \right) \rho(y)  \log \rho(y)
  + \left( \frac{\hat{x} - \delta - a}{\hat{x} - \eps - a} \right) \rho(y) \left( \frac{1 - 2 \frac{C_{\max}(\rmeas)}{C_{\min}(\rmeas)}}{\hat{x} - a - \eps}
  - O(\eps)\right)  \D y.
\end{align}
Therefore,
\begin{align}
   &\limsup_{\eps \to 0}
  \frac{1}{\eps}
    \int \limits_a^{\diff^*(\hat{x} - \eps)} \left( \frac{\hat{x} - \delta - a}{\hat{x} - \eps - a} - 1 \right) \rho(y) \log \rho(y)
      +
    \left( \frac{\hat{x} - \delta - a}{\hat{x} - \eps - a} \right) \rho(y) \log \left( \frac{\hat{x} - \delta - a}{\hat{x} - \eps - a} \right) \D y \\
    &\le \frac{\kappa - 1}{e} \frac{\diff^*(\hat{x}) - a}{\hat{x} - a}.
\end{align}
By assumption on $\hat{x}$, $\hat{x} - a > (b - a) / 2 > (\diff^*(\hat{x}) - a) / 2$.
Thus, $\limsup_{\eps \to 0} \clubsuit_{1} \le 2 \alpha \frac{\kappa - 1}{e}$.
To upper bound $\clubsuit_{2}$, note that $\frac{\hat{x} - \delta - a}{\hat{x} - \eps - a} \leq 1$, 
$R_{\eps}(\cdot) \le C_{\max} (\rmeas) / C_{\min}(\rmeas)$, and
\begin{align}
  \abs{R_{\eps}(y) - 1}
  &=
  \frac{
    \rmeas \left(\frac{\hat{x} - \delta - a}{\hat{x} - \eps - a}
    \left((\diff^*)^{-1}(y)-a\right) + a\right) 
    - \rmeas \left( (\diff^*)^{-1}(y) \right)
  }{
    \rmeas \left((\diff^*)^{-1}(y)\right)
  } 
  \le \frac{\norm{\partial_x \rmeas}_{L^\infty}}{C_{\min}(\rmeas)}
  \abs{ \frac{\delta - \eps}{\hat{x} - \eps - a} \left( (\diff^*)^{-1}(y) - a \right) } \\
  &\le  \frac{\norm{\partial_x \rmeas}_{L^\infty}}{C_{\min}(\rmeas)}
  \cdot (\kappa - 1) \eps 
  \le \frac{\norm{\partial_x \rmeas}_{L^{\infty}}}{C_{\min}(\rmeas)} \left( 2 \frac{C_{\max} (\rmeas)}{C_{\min}(\rmeas)} - 1 \right) \cdot \eps = C_{R} \eps,
\end{align}
where we have used the fact that $a \le y < \diff^*(\hat{x} - \eps)$ and thus $a \le (\diff^*)^{-1}(y) < \hat{x} - \eps$.
Also, using the fact that $|\log(1 + z)| \le 2 |z|$ for $\abs{z} \le 1/2$, we have for all $\eps > 0$ sufficiently small
\begin{align}
  \abs{\log R_{\eps}} 
  &\leq 2 \abs{\frac{\rmeas \left(\frac{\hat{x} - \delta - a}{\hat{x} - \eps - a}
    \left((\diff^*)^{-1}(y)-a\right) + a\right) 
    - \rmeas \left( (\diff^*)^{-1}(y) \right)
  }{
    \rmeas \left((\diff^*)^{-1}(y)\right)
  }} 
  \le 2C_R\eps.
\end{align}
Hence, we have
\begin{align}
  \limsup_{\eps \to 0} \clubsuit_2
  &\le
  \alpha C_R \int_a^b \abs{\rho(y)\log\rho(y)} \D y
  + 2 \alpha C_R \frac{C_{\max}(\rmeas)}{C_{\min}(\rmeas)}.
\end{align}
By optimality of $\diff^*$, using $F_{\alpha}^{\ugen,\rmeas}(\diff^*) \leq F_{\alpha}^{\ugen,\rmeas}(\mathrm{Id})$, $a \leq \diff^* \leq b$, $\norm{\ugen}_{L^\infty} \leq C_{\max}^{\mathrm{vel}}$, and $s \log s \geq -1/e$, we have
$\int_a^b \abs{\rho(y)\log\rho(y)} \D y \leq E_{\rho}$.
Thus,
\begin{equation}
  \limsup_{\eps \to 0} \clubsuit_2 \leq C_{\clubsuit_2}.
\end{equation}
Together with the estimate of $\clubsuit_{1}$,
we have
\begin{equation}
  \limsup_{\eps \to 0} \clubsuit
  \le C_{\clubsuit_2} +
  2 \alpha \frac{\kappa - 1}{e}.
\end{equation}
\textbf{(Step 3)}
In this step, we show that $\diamondsuit$ is sufficiently negative.
First, via a change of variable, note that
\begin{align}
    &\int \limits_{\diff^*(\hat{x} - \eps)}^{\diff^*(\hat{x})}
      \frac{ \delta }{\diff^*(\hat{x}) - \diff^*(\hat{x} - \eps)}
      \mu \left( \frac{\delta (y - \diff^*(\hat{x}- \eps))}{\diff^*(\hat{x}) - \diff^*(\hat{x} - \eps)} + \hat{x} - \delta \right) \\
    & \qquad \times \log \left( \frac{\delta}{\diff^*(\hat{x})
       - \diff^*(\hat{x} - \eps)} \mu \left( \frac{\delta (y - \diff^*(\hat{x}- \eps))}{\diff^*(\hat{x}) - \diff^*(\hat{x} - \eps)} + \hat{x} - \delta \right) \right) \\
    &= \int \limits_{\hat{x} - \delta}^{\hat{x}} \mu(x) \log \left( \frac{\delta}{\diff^*(\hat{x})
       - \diff^*(\hat{x} - \eps)}\mu(x) \right) \D x.
\end{align}
  By the Lebesgue differentiation theorem, we see that
  \begin{equation}
    \frac{1}{\diff^*(\hat{x}) - \diff^*(\hat{x} - \eps)} \int \limits_{\diff^{*}\left(\hat{x} - \eps\right)}^{\diff^{*}(\hat{x})} \rho(y) \D y \to \Lambda
    \textup{ as } \eps \to 0.
  \end{equation}
  Now, applying Jensen's inequality, we have
  \begin{align}
      &\frac{1}{\eps}
       \int \limits_{\hat{x} - \delta}^{\hat{x}} \mu(x) \log \left( \frac{\delta}{\diff^*(\hat{x})
       - \diff^*(\hat{x} - \eps)}\mu(x) \right) \D x -
      \frac{1}{\eps} \int \limits_{\diff^*(\hat{x} - \eps)}^{\diff^*(\hat{x})}
       \rho(y) \log \rho(y)
        \D y \\
      \le &
   \frac{\rmeas[\hat{x} - \delta, \hat{x}]}{\eps} \log \left( \frac{C_{\max} \delta}{\diff^*(\hat{x})
       - \diff^*(\hat{x} - \eps)}\right)
      - \frac{1}{\eps}  \int_{\diff^*(\hat{x} - \eps)}^{\diff^*(\hat{x})}
        \rho(y) \log \rho(y) \D y \\
      \le&
   \frac{\rmeas[\hat{x} - \delta, \hat{x}]}{\eps} \log \left( \frac{C_{\max} \delta}{\diff^*(\hat{x})
       - \diff^*(\hat{x} - \eps)}\right)  \\
    &-
      \frac{(\diff^*(\hat{x}) - \diff^*(\hat{x} - \eps))}{\eps}
      \left( \frac{1}{(\diff^*(\hat{x}) - \diff^*(\hat{x} - \eps))}
      \int \limits_{\diff^*(\hat{x} - \eps)}^{\diff^*(\hat{x})} \rho(y) \D
        y \right)     \!
       \log\! \left( \frac{1}{(\diff^*(\hat{x}) - \diff^*(\hat{x} - \eps))}
       \int \limits_{\diff^*(\hat{x} - \eps)}^{\diff^*(\hat{x})} \rho(y) \D y \right)
\end{align}
Recalling that $\rho = \frac{\dd \diff_{\#}^* \rmeas}{\dd \Leb}$, we see that
\begin{align}
      &\frac{(\diff^*(\hat{x}) - \diff^*(\hat{x} - \eps))}{\eps}
      \left( \frac{1}{(\diff^*(\hat{x}) - \diff^*(\hat{x} - \eps))}
      \int \limits_{\diff^*(\hat{x} - \eps)}^{\diff^*(\hat{x})} \rho(y) \D
        y \right)    \\
      &= \frac{1}{\eps}
      \int \limits_{\diff^*(\hat{x} - \eps)}^{\diff^*(\hat{x})} \rho(y) \D y
      = \frac{1}{\eps} \rmeas \left( (\diff^*)^{-1} [\diff^*(\hat{x} - \eps), \diff^*(\hat{x})] \right)
      = \frac{\rmeas[\hat{x} - \eps, \hat{x}]}{\eps}.
\end{align}
Therefore, we have
\begin{align}
      &\frac{1}{\eps}
       \int \limits_{\hat{x} - \delta}^{\hat{x}} \mu(x) \log \left( \frac{\delta}{\diff^*(\hat{x})
       - \diff^*(\hat{x} - \eps)}\mu(x) \right) \D x -
      \frac{1}{\eps} \int \limits_{\diff^*(\hat{x} - \eps)}^{\diff^*(\hat{x})}
       \rho(y) \log \rho(y)
        \D y \\
    &\le \frac{\rmeas[\hat{x} - \delta, \hat{x}]}{\eps} \log \left( \frac{C_{\max}\left(\rmeas\right) \delta}{\diff^*(\hat{x})
       - \diff^*(\hat{x} - \eps)}\right)
    -
    \frac{\rmeas[\hat{x} - \eps, \hat{x}]}{\eps}
       \cdot
       \log \left( \frac{1}{(\diff^*(\hat{x}) - \diff^*(\hat{x} - \eps))}
       \int \limits_{\diff^*(\hat{x} - \eps)}^{\diff^*(\hat{x})} \rho(y) \D y \right) \\
    &=  \frac{\rmeas[\hat{x} - \delta, \hat{x}]}{\eps}
    \cdot \log \left( \frac{C_{\max} (\rmeas) \delta}{\diff^*(\hat{x})
       - \diff^*(\hat{x} - \eps)}\right)
    -
    \frac{\rmeas[\hat{x} - \eps, \hat{x}]}{\eps}
       \cdot
       \log \left( \frac{\rmeas [\hat{x} - \eps, \hat{x}]}{(\diff^*(\hat{x}) - \diff^*(\hat{x} - \eps))}
       \right) \\
   &= \frac{\rmeas[\hat{x} - \delta, \hat{x}]}{\eps}
   \log \left( \frac{\rmeas [\hat{x} - \eps, \hat{x}]}{(\diff^*(\hat{x}) - \diff^*(\hat{x} - \eps))}\right)
   - \frac{\rmeas[\hat{x} - \eps, \hat{x}]}{\eps}
   \log \left( \frac{\rmeas [\hat{x} - \eps, \hat{x}]}{(\diff^*(\hat{x}) - \diff^*(\hat{x} - \eps))} \right) \\
   &\qquad + \frac{\rmeas[\hat{x} - \delta, \hat{x}]}{\eps}
   \log \left( C_{\max}(\rmeas) \frac{\delta}{\rmeas[\hat{x} - \eps, \hat{x}]} \right) \\
   &= \log \left( \frac{\rmeas [\hat{x} - \eps, \hat{x}]}{(\diff^*(\hat{x}) - \diff^*(\hat{x} - \eps))}\right)
   \cdot \left( \frac{\rmeas[\hat{x} - \delta, \hat{x}]}{\eps} -  \frac{\rmeas[\hat{x} - \eps, \hat{x}]}{\eps}   \right)
    - \frac{\rmeas[\hat{x} - \delta, \hat{x} ]}{\eps}
   \log \left( \frac{\rmeas[\hat{x} - \eps, \hat{x}]}{ \delta  \cdot C_{\max}(\rmeas)} \right).
\end{align}
By the Lebesgue differentiation theorem and $\delta = 2\frac{C_{\max}\left(\rmeas\right)}{C_{\min}\left(\rmeas\right)} \eps$,
we obtain the following lower bound on the second term in the equality above
\begin{align}
  \liminf_{\eps \to 0}
 \left[\frac{\rmeas[\hat{x} - \delta, \hat{x}]}{\eps}
   \log \left( \frac{\rmeas[\hat{x} - \eps, \hat{x}]}{ \delta  \cdot C_{\max}(\rmeas)} \right) \right]
   \ge \kappa C_{\max}(\rmeas) \log \left( \frac{C_{\min}(\rmeas)}{\kappa C_{\max}(\rmeas)} \right).
\end{align}
Also,
\begin{align}
  &\limsup_{\eps \to 0}
  \left[\log \left( \frac{\rmeas [\hat{x} - \eps, \hat{x}]}{(\diff^*(\hat{x}) - \diff^*(\hat{x} - \eps))}\right)
   \cdot \left( \frac{\rmeas[\hat{x} - \delta, \hat{x} ]}{\eps} -  \frac{\rmeas[\hat{x} - \eps, \hat{x}]}{\eps}   \right) \right] \\
  &\le \limsup_{\eps \to 0} C_{\max}\left(\rmeas\right) \left(\frac{C_{\min}(\rmeas)}{C_{\max}(\rmeas)} \frac{\delta}{\eps} - 1\right) \log \Lambda
  = C_{\max}\left(\rmeas\right) \log \Lambda.
\end{align}
The estimates in \textbf{Step 1}, \textbf{Step 2}, and \textbf{Step 3}, imply the existence of $0 < \eps_1 \le 1/4$ such that for all $\eps \le \eps_1$
\begin{equation}
  \heartsuit(\eps) + \clubsuit(\eps) \le \tilde{C} + C_{\clubsuit_2} + \eps_1
\end{equation}
and
\begin{equation}
  \diamondsuit(\eps) \le
\alpha \left(C_{\max}\left(\rmeas\right) \log \Lambda
  -  \kappa C_{\max}(\rmeas) \log \left( \frac{C_{\min}(\rmeas)}{\kappa C_{\max}(\rmeas)} \right) \right) + \eps_1.
\end{equation}
Therefore, for $\eps = \eps_1$,
\begin{align}
  &\frac{1}{\eps} \left( F_{\alpha}^{\ugen, \rmeas}(\diff(\cdot, \eps)) - F_{\alpha}^{\ugen, \rmeas}(\diff^*) \right)
  = \heartsuit(\eps) + \clubsuit(\eps) + \diamondsuit(\eps) \\
  &\le \tilde{C}(a, b, \norm{\ugen}_{L^{\infty}}, \rmeas)
  + C_{\clubsuit_2}
  + \alpha \left(C_{\max}\left(\rmeas\right) \log \Lambda
  - \kappa C_{\max}(\rmeas) \log \left( \frac{C_{\min}(\rmeas)}{\kappa C_{\max}(\rmeas)} \right) \right) + 2 \eps_1 \\
 & \le -1 + 1/2 = -1/2.
\end{align}
Hence, $\diff^*$ is not optimal.
A contradiction.
\end{proof}
\begin{remark}
Note that the proof of \cref{le:rho_lower_bound} also yields the following explicit bound on $\rho$ by tracing the constants above:
\begin{equation}
  \essinf_{y \in [a, b]} \rho(y) 
  \ge \exp\left(
  \frac{1}{\alpha C_{\max}(\rmeas)} 
  \left( - \tilde{C} -  \alpha C_R \left(E_{\rho} + 2\frac{C_{\max}(\rmeas)}{C_{\min}(\rmeas)}\right)
  + 2 \alpha \frac{C_{\max}^2(\rmeas)}{C_{\min}(\rmeas)} \log \left( \frac{1}{2}\left(\frac{C_{\min}(\rmeas)}{C_{\max}(\rmeas)} \right)^2\right)\right) \right),
\end{equation}
where 
$\tilde{C} \defeq
4 \frac{C_{\max}^2(\rmeas)}{C_{\min}(\rmeas)}(b - a)^2
+ 4 \frac{C_{\max}^2(\rmeas)}{C_{\min}(\rmeas)}  (b - a) \norm{\ugen}_{L^{\infty}}
+ 4 \frac{\alpha C_{\max}(\rmeas)}{e C_{\min}(\rmeas)}
$, 
\begin{align}
  C_R &\defeq \frac{\norm{\partial_x \rmeas}_{L^{\infty}}}{C_{\min}(\rmeas)} \left( 2 \frac{C_{\max} (\rmeas)}{C_{\min}(\rmeas)} - 1 \right), \\
  E_{\rho} &\defeq
  \frac{2\max\{\abs{a}, \abs{b}\} C_{\max}^{\mathrm{vel}}}{\alpha}
  + \abs{\Ent(\rmeas)}
  + \frac{2(b-a)}{e}.
\end{align}
\end{remark}
\subsection{Absolute continuity of \texorpdfstring{$\diff^*$}{the minimizer}}
\label{sec:absolute_continuity}
\begin{corollary}
\label{cor:minimizer_absolute_continuity}
Let $\rmeas$ be a probability measure satisfying $\rmeas \ll \mathscr{L}$ with $\mu \coloneqq \frac{\D \rmeas}{\D \Leb} \in W^{1,\infty}$, $0 < \essinf \mu$, $\esssup \mu < \infty$ and let $\ugen \in L^{\infty}$.
Let $\diff^*$ be the minimizer of $F_{\alpha}^{\ugen, \rmeas}$.
Then,
  \begin{enumerate}[label=(\roman*)]
    \item $\diff^*$ is bijective.
    \item $\diff^*$ is Lipschitz continuous and hence absolutely continuous ($\mathrm{a. c.}$) with uniformly bounded derivative.
    \item $\diff^*$ is the unique minimizer of $\hat{F}_{\alpha}^{\ugen, \rmeas}$ in $\Helly_{\mathrm{a.c.}} \coloneqq \{\diff \in \Helly: \diff \textup{ is a.c.}\}$.
 \end{enumerate}
 Thus, $F^{\ugen, \rmeas}_{\alpha}$ (hence $\hat{F}^{\ugen, \rmeas}_{\alpha}$) admits a unique minimizer in $W^{1, \infty}$.
\end{corollary}
\begin{proof}
  Let $\rho$ be as defined in \cref{le:rho_lower_bound}.

  (i) First, we show that $\diff^*$ is continuous.
  Towards a contradiction, suppose that $\diff^*$ has a discontinuity at $x_0 \in [a, b]$.
  Since $\diff^*$ is monotone, $\diff^*$ has a jump discontinuity at $x_0$.
  Then, there exists an interval $I \subseteq [a, b]$ such that $\int_I \rho(y) \D y = 0$,
  contradicting \cref{le:rho_lower_bound}.

  Next, we show that $\diff^*$ is bijective.
  The injectivity of $\diff^*$ is already proved in \cref{le:injective_minimizer}.
  Since $\diff^*(a) = a$ and $\diff^*(b) = b$ and the image of an interval under a continuous map remains an interval,
  we see that $\diff^*([a, b]) = [a, b]$, implying that $\diff^*$ is surjective.
  Therefore, $\diff^*$ is bijective.

  (ii) For any $a \leq x_1 < x_2 \leq b$, we have $\diff^*_{\#}\rmeas([\diff^*(x_1), \diff^*(x_2)]) < (x_2 - x_1) \esssup \rmeas $.
  Therefore, we have $\essinf_{y \in [\diff^*(x_1), \diff^*(x_2)]} \rho(y) \le \frac{x_2 - x_1}{\diff^*(x_2) -  \diff^*(x_1)}  \esssup \rmeas$. 
  Thus, by generality of $x_1, x_2$, $\diff^*$ is $\frac{\esssup \rmeas}{\essinf \rho}$-Lipschitz.
  The latter is finite by \cref{le:rho_lower_bound}.

  (iii) Uniqueness of $\diff^*$ in $\Helly_{\mathrm{a.c.}}$ follows from \cref{thm:uniqueness_minimizers}.
\end{proof}
\subsection{Lower bound on \texorpdfstring{$\partial_x \diff^*$}{the derivative of the minimizer}}
\label{sec:lower_bound_derivative}
Having established the weak differentiability of $\diff^*$ and thus minimization of $\hat{F}_{\alpha}^{\ugen, \rmeas}$, we aim to further analyze the regularity of $\diff^*$.
\begin{lemma}\label[lemma]{le:minimizer_lower_bound}
  Let $\diff^*$ be the minimizer of $F_{\alpha}^{\ugen, \rmeas}(\cdot)$ in $\Helly_{\textup{a.c.}}$ with $\mu \coloneqq \frac{\D \rmeas}{\D \Leb} \in W^{1,\infty}$, $0 < \essinf \mu \eqcolon C_{\min}(\rmeas)$ and $C_{\max}(\rmeas) \coloneq \esssup \mu < \infty$.
  Then, for $\norm{\ugen}_{L^{2}} \le \tilde{C}_{\max}^{\textup{vel}}$ there exists a $C_{\textup{low}}(\tilde{C}_{\max}^{\textup{vel}}, a, b, \alpha, C_{\min}(\rmeas), C_{\max}(\rmeas))$ satisfying
  \begin{equation}
    0 < C_{\mathrm{low}}(\tilde{C}_{\max}^{\textup{vel}}, a, b, \alpha, C_{\min}(\rmeas), C_{\max}(\rmeas)) \le
    \essinf_{x \in [a, b]} \partial_x \diff^*.
  \end{equation}
\end{lemma}
\begin{proof}
  Write $\mu(x) \coloneqq \frac{\D \rmeas}{\D \Leb}(x)$.
  Toward a contradiction, suppose that $\partial_x \diff^*$ is not almost everywhere uniformly lower bounded away from $0$ on $[a, b]$ and hence the same holds for $\frac{\partial_x \diff^*}{\mu}$.
  For any $\Lambda \ge 1$, define the sets
  \begin{equation}
    A = \left\{ \partial_x \diff^* \le \frac{\mu}{\Lambda} \right\}, \,
    B = \left\{ \partial_x \diff^* \ge \frac{2\mu}{\Lambda} \right\}, \,
    C = \left\{ \frac{\mu}{\Lambda} <  \partial_x \diff^* < \frac{2\mu}{\Lambda} \right\}.
  \end{equation}
  By assumption, $\mathscr{L}(A) > 0$.
  We claim that for all $\Lambda \ge 5 C_{\max}(\rmeas)$, the set $B$ has positive measure.
  Toward a contradiction, suppose that there exists $\Lambda \ge 5 C_{\max}(\rmeas)$ for which $B$ has measure zero.
  Since $\diff^*$ is absolutely continuous, we know that $\int_a^b \partial_x \diff^*(x) \D x = b - a$.
  Then, we have
  \begin{equation}
    \int \limits_a^b \partial_x \diff^*(x) \D x
    = \int \limits_A \partial_x \diff^*(x) \D x
      + \int \limits_C \partial_x \diff^*(x) \D x
    \le \frac{2 \rmeas([a, b])}{\Lambda}
    \le (b - a) \cdot \frac{2}{5} < (b - a) = \int \limits_a^b \partial_x \diff^*(x) \D x.
  \end{equation}
  A contradiction.

  Choose $\Lambda >
    \max \left\{\frac{2}{\alpha} \left( \norm{\diff^* - \mathrm{Id}}_{L^2(\rmeas)} + \norm{\ugen}_{L^2(\rmeas)} \right), 5 C_{\max}(\rmeas)\right\}.$
  We define $f\colon [a, b] \to \R$ by
  \begin{equation}
    f \colon x \mapsto \int \limits_a^x \frac{\rmeas(B)}{\rmeas(A)} \cdot \frac{1}{\Lambda}  \1_A(s) - \frac{1}{\Lambda} \1_B(s) \D \rmeas(s).
  \end{equation}
  Then, note that $f$ is absolutely continuous and $f(a) = f(b) = 0$.
  For all sufficiently small $\eps > 0$, note that $\diff^* + \eps f \in \Helly_{\mathrm{a.c.}}$, because for any $a \le x_1 < x_2 \le b$, $\partial_x \diff^* \geq C_{\min}(\rmeas) / \Lambda$ on $B$ implies that for $\eps < \frac{C_{\min}(\rmeas)}{\Lambda}$,
  \begin{equation}
    \partial_x \diff^* + \eps \partial_x f
    \ge \frac{C_{\min}(\mu)}{\Lambda} - \eps \cdot \frac{C_{\max}(\rmeas)}{\Lambda} \ge \frac{C_{\min}(\rmeas)}{\Lambda} - \frac{C_{\min}(\rmeas)}{5 \Lambda} > 0.
  \end{equation}
  Hence, 
  \begin{align}
    \left(\diff^* + \eps f \right)(x_2)
    - \left(\diff^* + \eps f \right)(x_1)
    = \int \limits_{x_1}^{x_2} \partial_x \diff^* + \eps \partial_x f \D x
    > 0.
  \end{align}
  Next, differentiating under the integral, we compute
  \begin{align}
    \frac{\D}{\D \eps}\bigg|_{\eps = 0} F_{\alpha}^{\ugen, \rmeas} (\diff^* + \eps f)
    &= \int \limits_a^b f \cdot  \left( \diff^* - \mathrm{Id} - \ugen \right) \D \rmeas(x)
    - \alpha \int\limits_a^b \frac{\partial_x f}{\partial_x \diff^*} \D \rmeas(x) \\
    &= \int \limits_a^b f \cdot (\diff^* - \mathrm{Id} - \ugen) \D \rmeas(x)
    - \alpha
    \left( \int \limits_A \frac{\partial_x f}{\partial_x \diff^*} \D \rmeas(x)+
    \int\limits_B \frac{\partial_x f}{\partial_x \diff^*} \D \rmeas(x) \right)  \\
    &\le \int \limits_a^b f \cdot (\diff^* - \mathrm{Id} - \ugen) \D \rmeas(x)
    - \alpha \left( \Lambda \rmeas(B)  \frac{1}{\Lambda} - \frac{\Lambda}{2} \frac{1}{\Lambda} \rmeas(B) \right) \\
    &\le \norm{\diff^* - \mathrm{Id} - \ugen}_{L^2(\rmeas)} \cdot \left( \int \limits_a^b f^2 \D \rmeas \right)^{1/2}
    - \frac{\alpha}{2} \cdot \rmeas(B) \\
    &\le \left( \norm{\diff^* - \mathrm{Id}}_{L^2(\rmeas)} + \norm{\ugen}_{L^2(\rmeas)} \right) \cdot   \frac{\rmeas(B)}{\Lambda} - \frac{\alpha}{2} \rmeas(B).
  \end{align}
  By choice of $\Lambda$,
  $\frac{\D}{\D \eps}\big|_{\eps = 0} F_{\alpha}^{\ugen, \rmeas} (\diff^* + \eps f) < 0$.
  Thus, there exists $\eps > 0$ such that
    $F_{\alpha}^{\ugen, \rmeas}(\diff^* + \eps f) < F_{\alpha}^{\ugen, \rmeas}(\diff^*)$,
  contradicting the optimality of $\diff^*$.
  Therefore, $\partial_x \diff^*$ must be uniformly lower bounded by a constant $C_{\text{low}}$.
  In fact, the above argument produces an explicit lower bound on $\partial_x \diff^*$, namely,
  \begin{align}
    \essinf_{x \in [a, b]} \partial_x \diff^*(x)
    &\ge
  \min \left\{ \frac{\alpha C_{\min}(\rmeas)}{2 \left( \norm{\diff^* - \mathrm{Id}}_{L^2(\rmeas)} +  \norm{\ugen}_{L^2(\rmeas)} \right)}, \frac{C_{\min}(\rmeas)}{5 C_{\max}(\rmeas)} \right\} \\
    &\ge
    \min \left\{ \frac{\alpha C_{\min}(\rmeas)}{2 \left( (b - a) +  \norm{\ugen}_{L^2(\rmeas)} \right)}, \frac{C_{\min}(\rmeas)}{5 C_{\max}(\rmeas)} \right\}.
  \end{align}
  This concludes the proof.
\end{proof}
\subsection{Higher-order regularity}
\label{sec:higher_order_regularity}
We prove higher regularity of $ \diff^*$ by a bootstrapping argument, based on the following optimality condition.
\begin{lemma}\label[lemma]{lem:optimality_condition}
  For $\mu \coloneqq \frac{\D \rmeas}{\D \Leb} \in W^{1,\infty}$ with $0 < \essinf \mu \leq \esssup \mu < \infty$ and $\ugen \in L^{\infty}$, $\diff^*$ is the unique element of $\helly \cap W^{1, \infty}$ satisfying, for all $\varphi \in C^{\infty}_{0}$,
  \begin{equation}
    \int \limits_a^b \varphi(x) \diff^*(x)  - \alpha \frac{\partial_x \varphi }{\partial_x \diff^*}\D \rmeas(x)
    =  \int \limits_a^b \varphi(x) \left( x + \ugen(x)\right) \D \rmeas(x).
  \end{equation}
\end{lemma}
\begin{proof}
  Let $\varphi \in C^{\infty}_{0}([a, b])$ be general.
  By \cref{le:minimizer_lower_bound,cor:minimizer_absolute_continuity}, we have $\diff^*, \diff^* + \eps \varphi \in \helly \cap W^{1, \infty}$ for all sufficiently small $\eps > 0$.
  By optimality of $\diff^*$, we have
  \begin{align}
  0 \leq &\int \limits_{a}^{b} \frac{\left(\left(\diff^*(x) + \eps \varphi(x)\right) - x\right)^2}{2} - \alpha \log\left(\partial_x \diff^* + \eps \partial_x \varphi\right) - \left(\diff^*(x) + \eps \varphi(x)\right) \ugen(x) \D \rmeas(x) \\
  -&\int \limits_{a}^{b} \frac{\left(\left(\diff^*(x)\right) - x\right)^2}{2} - \alpha \log\left(\partial_x \diff^* \right) - \diff^*(x) \ugen(x) \D \rmeas(x).
  \end{align}
  We can estimate
  \begin{equation}
    \frac{\norm{\left(\diff^* + \eps \varphi\right) - \mathrm{Id}}_{L^2(\rmeas)}^2}{2} - \frac{\norm{\left(\diff^*\right) - \mathrm{Id}}_{L^2(\rmeas)}^2}{2} = \eps \inprod{ \varphi , \diff^* - \mathrm{Id} }_{L^2(\rmeas)} + \eps^2 \frac{1}{2} \norm{\varphi}_{L^2(\rmeas)}^2
  \end{equation}
  and, by Taylor expansion of $s \mapsto \log(s)$,
  \begin{equation}
    \left\|\log\left(\partial_x \diff^* + \eps \partial_x \varphi\right) - \log\left(\partial_x \diff^*\right) - \eps \frac{\partial_x \varphi}{\partial_x \diff^*} \right\|_{L^2(\rmeas)}
    \leq  \eps^2 C(\sup \partial_x \varphi, \inf \partial_x \diff^*).
  \end{equation}
  Thus, by the Dominated Convergence Theorem, we see that
  \begin{align}
    \lim_{\eps \to 0} \frac{1}{\eps} &\left[
      \begin{multlined}
     \int \limits_{a}^{b} \frac{\left(\left(\diff^*(x) + \eps \varphi(x)\right) - x\right)^2}{2} - \alpha \log\left(\partial_x \diff^* + \eps \partial_x \varphi\right) - \left(\diff^*(x) + \eps \varphi(x)\right) \ugen(x) \D \rmeas(x) \\
      -\int \limits_{a}^{b} \frac{\left(\left(\diff^*(x)\right) - x\right)^2}{2} - \alpha \log\left(\partial_x \diff^* \right) - \diff^*(x) \ugen(x) \D \rmeas(x)
      \end{multlined}
    \right] \\
    =&
    \int \limits_a^b \varphi(x) \diff^*(x)  - \alpha \frac{\partial_x \varphi }{\partial_x \diff^*}\D \rmeas(x)
    -  \int \limits_a^b \varphi(x) \left( x + \ugen(x)\right) \D \rmeas(x)
    = 0.
  \end{align}
\end{proof}
\begin{lemma} \label[lemma]{lem:representation_of_minimizer}
  Under the assumptions of \cref{lem:optimality_condition}, there exists a constant $c_{\ugen} = c_{\ugen}(\alpha, \diff^*, \ugen, \rmeas)$ such that
  \begin{equation}
    \frac{1}{\partial_x \diff^*(x)} = \frac{c_{\ugen}}{\mu(x)} + \frac{1}{\alpha \mu(x)} \int \limits_{a}^{x} \ugen(s) - \left(\diff^*(s) - s \right) \D \rmeas(s)
    \quad \textup{a.e.}
  \end{equation}
\end{lemma}
\begin{proof}
  By \cref{lem:optimality_condition}, we see that for any $\varphi \in C_0^{\infty}$,
  \begin{equation}
    \int \limits_a^b \frac{\partial_x \varphi(x)}{\partial_x \diff^*(x)} \mu(x) \D x
    =
    - \frac{1}{\alpha}
    \int \limits_a^b \varphi(x) \cdot \left( \ugen(x) - \left( \diff^*(x) - x \right) \right)
    \mu(x) \D x.
  \end{equation}
  Therefore, by the definition of weak derivative, we see that
  \begin{equation}
    \partial_x \left(\frac{\mu(x)}{\partial_x \diff^*(x)} \right) =
    \frac{1}{\alpha} \left( \ugen(x) - (\diff^*(x) - x) \right) \mu(x).
  \end{equation}
  By the uniqueness of weak derivative, we therefore have
  \begin{equation}
    \frac{1}{\partial_x \diff^*(x)} = \frac{c(\ugen, \alpha, \diff^*, \rmeas)}{\mu(x)} + \frac{1}{\alpha \mu(x)} \int \limits_{a}^x
    \ugen(s) - \left( \diff^*(s) - s \right) \D \rmeas(s) \textup{ a.e.}
  \end{equation}
  for some constant $c(\ugen, \alpha, \diff^*, \rmeas)$ depending only on $\ugen$, $\alpha$, $\diff^*$, and $\rmeas$.
\end{proof}
This yields higher regularity of $\diff^*$ as follows.
\begin{lemma}
  \label[lemma]{lem:higher_regularity}
  Under the assumptions of \cref{lem:optimality_condition}, let $\diff^*$ be the minimizer of $F_{\alpha}^{\ugen,\rmeas}(\cdot)$ in $\Helly_{\mathrm{a.c.}}$.
  Let $k \geq 0$ and suppose additionally that $\ugen \in W^{k, \infty}$ and $\mu \in W^{k + 1, \infty}$.
  Then, we have $\diff^* \in W^{k + 2, \infty}$.
  An analogous result holds with $W^{k, \infty}, W^{k + 1, \infty}, W^{k + 2, \infty}$ replaced by $C^k, C^{k + 1}, C^{k + 2}$.
\end{lemma}
\begin{proof}
  By \cref{lem:representation_of_minimizer} and the fact that $\abs{\diff^*}$ and $\mu(\cdot)$ are uniformly bounded, we see that
  \begin{equation}
    \label{eq:representation_of_partial_diff}
    \partial_x \diff^* (x)
    =
    \frac{\mu(x)}{c(\ugen, \alpha, \diff^*, \rmeas) + \frac{1}{\alpha} \int \limits_a^x \ugen(s) - \left( \diff^*(s) - s \right) \D \rmeas(s)}
  \end{equation}
  By \cref{cor:minimizer_absolute_continuity,le:minimizer_lower_bound}, both $\partial_x\diff^*$ and $1/\partial_x\diff^*$ are bounded, so the denominator in \eqref{eq:representation_of_partial_diff} is bounded away from zero.
  By $\ugen \in L^{\infty}$, $\mu(\cdot) \in W^{1, \infty}$, and $\diff^* \in W^{1, \infty}$, we have $\partial_x \diff^* \in W^{1, \infty}$.
  Hence, $\diff^* \in W^{2, \infty}$.
  Iterating this argument yields the desired result. The proof for $C^k, C^{k+1}, C^{k + 2}$ is analogous.
\end{proof}
We note that for the functional $f^{t u_0}_{\alpha}$ in \cref{def:variational_solution} to be represented by a $W^{k, \infty}$ function, we require that $u_0 \in W^{k + 2, \infty}$.
Thus, our higher regularity result implies that in this case, the regularity of the deformation $\diff^{*}$ matches that of the initial velocity field $u_0$ for sufficiently regular $\rmeas$.

\section{Stability, time derivatives, and Eulerian solutions}
\label{sec:eulerian}
\subsection*{Overview}
In this section we prove the stability and differentiability of minimizers of $F_{\alpha}^{\ugen, \rmeas}$ with respect to the problem data, in higher-order Sobolev norms. 
This allows establishing the existence of smooth Lagrangian and Eulerian solutions of pressureless IGR.
In \Cref{sec:higher_order_stability}, we prove the stability of minimizers in $W^{k, \infty}$ and $C^k$.
\Cref{sec:time_derivative} shows the differentiability of minimizers in these norms, allowing us to prove the existence of time derivatives.
\Cref{sec:first_order_pde_solution} uses these results to show that the map $t \mapsto \diff_t$ satisfies the Lagrangian pressureless IGR equation \cref{eqn:lagrangian_eqn} and concludes the existence of smooth solutions to the Eulerian pressureless IGR equation \cref{eqn:1d-igr-pressureless}.
\subsection{Stability of minimizers}
\label{sec:higher_order_stability}
Throughout this section, we assume that the probability measure $\rmeas$ on $[a,b]$ is absolutely continuous with respect to $\Leb$ and denote as $\mu(\cdot)$ its Lebesgue density.
We assume that $0 < \essinf_{x \in [a, b]} \rmeas(x)$, $\esssup_{x \in [a, b]} \rmeas(x) < \infty$, and $\rmeas \in W^{1, \infty}([a, b])$.
\begin{lemma}
  \label[lemma]{lem:lipschitzH1} 
  The map $\ugen \mapsto \diff^{\ugen} \coloneq \argmin_{\diff \in \helly} F_{\alpha}^{\ugen, \rmeas}$ is Lipschitz continuous as a map from $L^2$ to $H^1$, within each $L^{\infty}$-ball of finite radius. 
\end{lemma}
\begin{proof}
  Consider an $L^\infty$-ball of radius $R$. 
  By \cref{le:rho_lower_bound,le:minimizer_lower_bound}, there exist constants $C_{\min}$ and $C_{\max}$ depending only on $\alpha, \rmeas, R, a, b$ such that for all $\ugen \in L^{\infty}([a, b]), \left\|\ugen\right\|_{L^{\infty}} \leq R$, we have $C_{\min} \leq \partial_x \diff^{\ugen} \leq C_{\max}$.
  On the interval $[C_{\min}, C_{\max}]$, the function $z \mapsto - \log(z) - z^2 / (2C_{\max}^2)$ is convex.
  Thus, the function 
  \begin{equation}
      \tilde{F}(\diff) \coloneqq 
      \begin{cases}
          \int \limits_{a}^{b} \frac{\left(\diff(x) - x\right)^2}{2} - \alpha \log\left(\partial_x \diff(x)\right) \D \rmeas(x), & \text{if } \diff \in \Helly_{\mathrm{a.c.}}, C_{\min} \leq \partial_x \diff \leq C_{\max},\  \rmeas\operatorname{-a.e.}\\
          \hfill \infty, \hfill & \text{else}.
      \end{cases}
  \end{equation}
  is $\alpha / C_{\max}^2$-strongly convex with respect to $H^1({\rmeas})$ and satisfies 
\begin{equation}
  \diff^{\ugen} = \argmin_{\diff} \tilde{F}(\cdot) - \langle \mathcal{R}(\ugen), \cdot \rangle_{H^1(\rmeas)},
\end{equation}
  where $\mathcal{R}(\ugen)$ is the Riesz-representer of the map $\diff \mapsto \int_{a}^{b} \ugen(x) \diff(x) \D \rmeas(x)$. 
  By \cref{lem:lipschitz_cont_solution_map}, applied to the pair $\mathcal{R}(u_1), \mathcal{R}(u_2)$, for any $u_1, u_2 \in L^{\infty}([a, b])$ with $\|u_1\|_{L^{\infty}}, \|u_2\|_{L^{\infty}} \leq R$, the map $\mathcal{R}(\ugen) \mapsto \diff^{\ugen}$ is Lipschitz continuous in $H^1$.
  Moreover, we have
  \begin{align}
  &\|\mathcal{R}(u_2) - \mathcal{R}(u_1)\|_{H^1(\rmeas)}^2 = \left\langle \mathcal{R}(u_2) - \mathcal{R}(u_1), u_2 - u_1 \right\rangle_{L^2\left(\rmeas\right)} 
  \leq
  \left\|\mathcal{R}(u_2) - \mathcal{R}(u_1)\right\|_{L^2\left(\rmeas\right)} \left\|u_2 - u_1\right\|_{L^2\left(\rmeas\right)}\\
  \implies & \|\mathcal{R}(u_2) - \mathcal{R}(u_1)\|_{H^1(\rmeas)} 
    \leq \left\|u_2 - u_1\right\|_{L^2\left(\rmeas\right)},
  \end{align}
  from which the claim follows.
\end{proof}
\begin{lemma}
  \label[lemma]{lem:sobolev_stability_of_minimizers}
  Let $k \geq 0$. Assuming $\rmeas(\cdot)$ is $W^{k + 1,\infty}$, the map $\ugen \mapsto \diff^{\ugen}$ is locally Lipschitz continuous as a map from $W^{k, \infty}$ to $W^{k + 2, \infty}$. The same result holds when replacing $W^{k, \infty}, W^{k + 1, \infty}, W^{k + 2, \infty}$ with $C^k, C^{k + 1}, C^{k + 2}$.
\end{lemma}
\begin{proof}
  By \cref{lem:representation_of_minimizer}, we have for $R>0$, $u_1, u_2 \in L^{\infty}(\rmeas), \left\|u_1\right\|_{L^{\infty}}, \left\|u_2\right\|_{L^{\infty}} \leq R$, and $\diff_1 = \diff^{u_1}, \diff_2 = \diff^{u_2}$,
  \begin{align}
    \partial_x \diff_1(x) - \partial_x \diff_2(x) 
    &= \frac{\rmeas(x)}{c_{u_1} + \frac{1}{\alpha} 
    \int \limits_{a}^{x} u_1(s) - \left(\diff_1(s) - s\right) \D \rmeas(s)} - \frac{\rmeas(x)}{c_{u_2} 
    + \frac{1}{\alpha} \int \limits_{a}^{x} u_2(s) - \left(\diff_2(s) - s\right) \D \rmeas(s)} \\ 
    = & \frac{\left(c_{u_2} - c_{u_1} \right) + \frac{1}{\alpha} \int \limits_{a}^{x} \left(u_2(s) - u_1(s)\right) - \left(\diff_2(s) - \diff_1(s)\right) \D \rmeas(s) }{\rmeas(x) / \left[\left(\partial_x\diff_1\left(x\right)\right)\left(\partial_x \diff_2\left(x\right)\right)\right]},
  \end{align}
  using that  $c_{u_i} + \frac{1}{\alpha}\int_a^x \cdots \D\rmeas = \rmeas(x)/\partial_x \diff_i(x)$, and thus the denominators' product is $\rmeas(x)^2 \big/ \left[\left(\partial_x \diff_1\right)\left(\partial_x \diff_2\right)\right]$.
  Rearranging terms and using the uniform upper and lower bounds $\rmeas \le C_{\max}(\rmeas)$, $\partial_x \diff_i \ge C_{\min}$, we have
  \begin{equation}
    \left|c_{u_2} - c_{u_1}\right| \leq \frac{C_{\max}(\rmeas)}{C_{\min}^2} \left|\partial_x \diff_1(x) - \partial_x \diff_2(x)\right| + \frac{1}{\alpha} \left(\|u_2 - u_1\|_{L^{\infty}} + \left\|\diff_2 - \diff_1 \right\|_{L^{\infty}} \right).
  \end{equation}
  Upon integrating both sides over $[a, b]$, \cref{lem:lipschitzH1},
  together with the Sobolev embedding theorem from $H^1$ to $C^{1/2}$, implies that the function $u \mapsto c_{u}$ is locally Lipschitz continuous as a map from $L^{\infty}(\rmeas)$ to $\R$. 
  Thus, 
  \begin{align}
    \left|\partial_x \diff_1(x) - \partial_x \diff_2(x)\right| 
    &\leq  \frac{\left|c_{u_2} - c_{u_1}\right| + \frac{1}{\alpha } \left(\|u_2 - u_1\|_{L^{\infty}} 
+ \left\|\diff_2 - \diff_1 \right\|_{L^{\infty}} \right)}{C_{\min}(\rmeas)} \\
    &\leq L(a,b, \alpha, R, \rmeas) \|u_2 - u_1\|_{L^{\infty}}
  \end{align}
  for a Lipschitz constant $L$ depending on $a, b, \alpha, R, \rmeas$ and $\mu \geq C_{\min}(\mu)$.
  We now take the spatial derivative of 
  \begin{equation}
    \label{eqn:difference_of_spatial_derivatives}
    \partial_x \diff_1(x) - \partial_x \diff_2(x)
    = \frac{\rmeas(x)}{c_{u_1} + \frac{1}{\alpha} \int \limits_{a}^{x} u_1(s) - \left(\diff_1(s) - s\right) \D \rmeas(s)} - \frac{\rmeas(x)}{c_{u_2} + \frac{1}{\alpha} \int \limits_{a}^{x} u_2(s) - \left(\diff_2(s) - s\right) \D \rmeas(s)} 
  \end{equation}
  to obtain
  \begin{align}
    \partial_x^2 \diff_1(x) - \partial_x^2 \diff_2(x) 
    &= - \frac{\frac{\rmeas(x)^2}{\alpha} \left(u_1(x) - \left(\diff_1(x) - x\right)\right) - \partial_x \rmeas(x)\left(c_{u_1} +  \frac{1}{\alpha}\int \limits_a^x u_1(s) - (\diff_1(s) - s) \D \rmeas(s)\right)}{\left(c_{u_1} + \frac{1}{\alpha} \int \limits_{a}^{x} u_1(s) - \left(\diff_1(s) - s\right) \D \rmeas(s) \right)^2} \\
    &\quad +
    \frac{\frac{\rmeas(x)^2}{\alpha} \left(u_2(x) - \left(\diff_2(x) - x\right)\right) - \partial_x \rmeas(x) \left(c_{u_2} + \frac{1}
    {\alpha}\int \limits_a^x u_2(s) - (\diff_2(s) - s) \D \rmeas(s)\right)}{\left(c_{u_2} + \frac{1}{\alpha} \int \limits_{a}^{x} u_2(s) - \left(\diff_2(s) - s \right) \D \rmeas(s) \right)^2} \\
    &= - \left(\!\! \frac{\left(u_1(x) - \left(\diff_1(x) - x\right)\right)}{\alpha}  \!-\! \frac{\partial_x \rmeas(x) \left(\!c_{u_1} + \frac{1}{\alpha}\int \limits_a^x u_1(s) - (\diff_1(s) - s) \D \rmeas(s)\! \right)}{\rmeas(x)^2}\!\!\right)\!\left( \partial_x \diff_1(x) \right)^2 \\
    &\quad +
    \left(\!\! \frac{\left(u_2(x) - \left(\diff_2(x) - x\right)\right)}{\alpha}  \!-\! \frac{\partial_x \rmeas(x) \left(\!c_{u_2} + \frac{1}{\alpha}\int \limits_a^x u_2(s) - (\diff_2(s) - s) \D \rmeas(s)\! \right)}{\rmeas(x)^2}\!\!\right) \! \left( \partial_x \diff_2(x) \right)^2 
    %
  \end{align}
  Here, we have again used the representation of $\partial_x \diff$ from \cref{lem:representation_of_minimizer}.
  For prefactors $a_1, a_2$, we have
  \begin{align}
    a_2 (\partial_x \diff_2)^2 - a_1  (\partial_x \diff_1)^2
    =
    (a_2 - a_1) (\partial_x \diff_2)^2 
    +
    a_1 (\partial_x \diff_2 + \partial_x \diff_1)(\partial_x \diff_2 - \partial_x \diff_1)  
  \end{align}
 Therefore, using the fact that $\partial_x \mu$, $\mu$, $\partial_x \diff_1$,
  and $\partial_x \diff_2$ are all uniformly upper and lower bounded,
  we obtain
  \begin{equation}
    \left|\partial_x^2 \diff_1(x) - \partial_x^2 \diff_2(x)\right| 
    \leq  L(a,b, \alpha, R, \rmeas, \partial_x \rmeas) \|u_1 - u_2\|_{L^{\infty}}
  \end{equation}
  for a possibly larger $L$, thus establishing the claim. 
  The same argument can be made to show that the map $u \mapsto \diff^{u}$ is locally Lipschitz continuous as a map from $C^0$ to $C^{2}$.
  The proof for $k > 0$ follows by induction on $k$, by taking successively higher-order $x$-derivatives of \cref{eqn:difference_of_spatial_derivatives}.
\end{proof}
\subsection{Time derivatives}
\label{sec:time_derivative}
Let $\diff_t$ be the minimizer of $F_{\alpha}^{t \uvel, \rmeas}(\cdot)$ in $\Helly_{\mathrm{a.c.}}$.
We will now use the optimality condition of the minimization problem to prove the existence of the time derivative $\dot{\diff}_t$ of $\diff_t$.
As in the previous section, we assume that the probability measure $\rmeas$ on $[a,b]$ is absolutely continuous with respect to $\Leb$ and denote as $\mu(\cdot)$ its Lebesgue density satisfying $0 < \essinf_{[a, b]} \rmeas$, $\esssup_{[a, b]} \rmeas < \infty$, and $\mu\in W^{1,\infty}$.
\begin{lemma}
  \label[lemma]{lem:elliptic_pde_solution}
  For $k \geq 0$, let $\diff \in W^{k + 2, \infty}$ with $\partial_x \diff \geq C_{\min} > 0$ and $\norm{\diff}_{W^{k + 2, \infty}} \leq R$, $\mu\in W^{k+1,\infty}$, and $\ugen \in W^{k, \infty}$.
  Then, the weak solution $\theta$ of the elliptic PDE 
  \begin{equation}
    \int \varphi(x) \theta(x) + \alpha \frac{\partial_x\varphi(x) \partial_x \theta}{\left(\partial_x \diff \right)^2}  \D \rmeas(x) = \int \varphi(x) \ugen(x) \D \rmeas(x)
  \end{equation} 
  is in $W^{k + 2, \infty}$ and there exists a $C(a, b, \alpha, R, C_{\min})$ such that $\norm{\theta}_{W^{k + 2, \infty}} \leq C(a, b, \alpha, R, C_{\min}) \norm{\ugen}_{W^{k, \infty}}$.
  The same result holds when replacing $W^{k, \infty}, W^{k + 2, \infty}$ with $C^k, C^{k + 2}$.
\end{lemma}
\begin{proof}
  The Lax-Milgram theorem ensures the existence of a solution in $H^1(\rmeas)$, with 
  \begin{equation}
    \norm{\theta}_{H^1(\rmeas)} \leq C(a, b, \alpha, R, C_{\min}, \rmeas) \norm{\ugen}_{L^{\infty}}.
  \end{equation}
  Similar to \cref{lem:representation_of_minimizer}, we have 
  \begin{equation}
    \label{eqn:representation_of_theta}
    \partial_x \theta(x) = \frac{c_{\ugen, \diff}\left(\partial_x \diff\right)^2}{\rmeas(x)} - \frac{\left(\partial_x \diff\right)^2}{\alpha \rmeas(x)} \int \limits_{a}^{x} \ugen(s) - \theta(s) \D \rmeas(s).
  \end{equation}
  Rearranging terms and integrating over $[a, b]$ yields $c_{\ugen, \diff} \leq C(a, b, \alpha, R, \rmeas) \norm{\ugen}_{L^\infty}$. 
  Differentiation yields 
  \begin{align}
    \partial^2_x \theta(x) 
    &= \left(\frac{2\partial_x \diff(x) \partial_x^2 \diff(x)}{\rmeas(x)} - \frac{\left( \partial_x \diff \right)^2 \partial_x \mu(x)}{\rmeas(x)^2}\right)\left(c_{\ugen, \diff} - \frac{1}{\alpha}\int \limits_{a}^{x} \ugen(s) - \theta(s) \D \mu(s) \right)\\
    &\quad - \frac{\left(\partial_x \diff(x)\right)^2}{\alpha} \left(\ugen(x) - \theta(x)\right) \\
    &\leq C(a, b, \alpha, R, \rmeas) \norm{\ugen}_{L^\infty},
  \end{align}
  from which the result follows. The proof with $C^0$ and $C^2$ is analogous.
  The result for $k > 0$ follows by a bootstrapping argument, taking successive $x$-derivatives of \cref{eqn:representation_of_theta}.
\end{proof}
We now prove the existence and regularity of the time derivatives of $\diff^*_t$.
\begin{theorem}
  \label{thm:time_derivative}
  For $k \geq 0$, let $\uvel, w \in W^{k, \infty}$ and $\rmeas \in W^{k + 1, \infty}$.
  For $t \in \R$, let $\diff_t$ be the minimizer of $F_{\alpha}^{w + t \uvel, \rmeas}(\cdot)$ in $\Helly_{\mathrm{a.c.}}$.
  Then, the weak solution $\dot{\diff}_t$ of the elliptic PDE 
  \begin{equation}\label{eqn:time_derivative_elliptic}
    \int \limits_a^b \varphi(x) \dot{\diff}_t(x) + \alpha \frac{\partial_x\varphi(x) \partial_x \dot{\diff}_t}{\left(\partial_x \diff_t\right)^2}  \D \rmeas(x) 
    = \int \limits_a^b \varphi(x) \uvel(x) \D \rmeas(x)
  \end{equation} 
  exists in $W^{k + 2, \infty}$ for all $t \in \R$ and satisfies 
  \begin{equation}
    \lim \limits_{\delta_t \to 0} \frac{\left\|\diff_{t + \delta_t} - \diff_{t} - \delta_t \dot{\diff}_t\right\|_{W^{k,\infty}}}{\delta_t}  = 0.
  \end{equation}
  The same result holds when replacing $W^{k, \infty}, W^{k + 1, \infty} ,W^{k + 2, \infty}$ with $C^k, C^{k + 1}, C^{k + 2}$.
\end{theorem}
\begin{proof}
  From \cref{lem:elliptic_pde_solution}, we already know that the elliptic PDE in \cref{eqn:time_derivative_elliptic} admits a solution $\theta \in W^{k + 2, \infty}$. 
  To show that this solution is the time derivative $\dot{\diff}_t$ of $\diff_t$, we now show that
  \begin{equation}
    \lim \limits_{\delta_t \to 0} \frac{\left\|\diff_{t + \delta_t} - \diff_{t} - \delta_t \theta \right\|_{W^{k,\infty}}}{\delta_t}  = 0.
  \end{equation}
  We begin by proving the result for $k=0$. By the optimality condition in \cref{lem:optimality_condition}, we have 
  \begin{align}
    \int \limits_a^b \varphi(x) \left(\diff_{t + \delta_t}(x) - \diff_t(x)\right)  -  \alpha \partial_x \varphi(x) \left(\frac{1}{\partial_x \diff_{t + \delta_t}(x)} - \frac{1}{\partial_x \diff_{t}(x)}\right)\D \rmeas(x)
  =  \delta_t \int \limits_a^b \varphi(x) \uvel(x) \D \rmeas(x).
  \end{align}
  Adding an empty sum yields 
  \begin{align}
  & \int \limits_a^b \varphi(x) \left(\diff_{t + \delta_t}(x) - \diff_t(x)\right) \\ 
    &\qquad -  \alpha \partial_x \varphi(x) \left(- \frac{\partial_x \diff_{t + \delta_t}(x) - \partial_x \diff_{t}(x)}{\left(\partial_x \diff_{t} (x)\right)^2}
    + \frac{1}{\partial_x \diff_{t + \delta_t}(x)} - \frac{1}{\partial_x \diff_{t}(x)} + \frac{\partial_x \diff_{t + \delta_t}(x) - \partial_x \diff_{t}(x)}{\left(\partial_x \diff_{t} (x)\right)^2}  \right) \D \rmeas(x) \\
  &  \qquad \qquad =  \delta_t \int \limits_a^b \varphi(x) \uvel(x) \D \rmeas(x).
  \end{align}
  Rearranging terms and integrating by parts, we obtain (writing $\rmeas(x)$ for the Lebesgue density of $\rmeas$)
  \begin{align}
  & \int \limits_a^b \left[ \varphi(x) \left(\diff_{t + \delta_t}(x) - \diff_t(x)\right)
    +  \alpha \partial_x \varphi(x) \frac{\partial_x \diff_{t + \delta_t}(x) - \partial_x \diff_{t}(x)}{\left(\partial_x \diff_{t} (x)\right)^2} \right] \D \rmeas(x) \\
  &\quad + \alpha \int \limits_a^b \varphi(x) \partial_x \left(
      \rmeas(x) \frac{\left(\partial_x \diff_{t + \delta_t}(x) - \partial_x \diff_{t}(x)\right)^2}{\partial_x \diff_{t + \delta_t}(x) \left(\partial_x \diff_{t} (x)\right)^2}
      \right) \D x
   =  \delta_t \int \limits_a^b \varphi(x) \uvel(x)  \D \rmeas(x).
  \end{align}
  Since $\rmeas(x)$ now sits inside the derivative, the remainder term is integrated against $\D x = \D \Leb$. Defining
  \begin{equation}
    \mathrm{Rhs}_{\delta_t}(x)
    \coloneqq
    - \frac{\alpha}{\rmeas(x)} \partial_x \left( \rmeas(x) \frac{\left(\partial_x \diff_{t + \delta_t}(x) - \partial_x \diff_{t}(x)\right)^2}{\partial_x \diff_{t + \delta_t}(x) \left(\partial_x \diff_{t} (x)\right)^2} \right),
  \end{equation}
  so that this remainder term equals $- \int_a^b \varphi(x) \mathrm{Rhs}_{\delta_t}(x) \D \rmeas(x)$, we thus obtain
  \begin{align}
    \int \limits_a^b \varphi(x) \left(\diff_{t + \delta_t}(x) - \diff_t(x)\right)
    + \alpha \partial_x \varphi(x) \frac{\partial_x \diff_{t + \delta_t}(x) - \partial_x \diff_{t}(x)}{\left(\partial_x \diff_{t} (x)\right)^2} \D \rmeas(x)
    =  \int \limits_a^b \delta_t \varphi(x) \uvel(x) + \varphi(x) \mathrm{Rhs}_{\delta_t}(x) \D \rmeas(x).
  \end{align}
  Using the definition of $\dot\diff_t$, we obtain
  \begin{align}
    \int \limits_a^b \!\! \varphi(x) \left(\diff_{t + \delta_t} - \diff_t(x) - \delta_{t} \dot{\diff}_t(x)\right) 
    \!+\!  \alpha \frac{\partial_x \varphi(x) \partial_x \left(\diff_{t + \delta_t}(x) - \diff_{t}(x) - \delta_t \dot{\diff}_t(x)\right)}{\left(\partial_x \diff_{t} \right)^2}  \D \rmeas(x)
    \!=\!  \int \limits_a^b \!\! \varphi(x) \mathrm{Rhs}_{\delta_t}(x)\D\rmeas(x). 
  \end{align}
  The quantities $\partial_{x} \diff_{t}(x)$ and $\partial_{x} \diff_{t + \delta_t}(x)$ are uniformly upper and lower bounded by \cref{le:rho_lower_bound,le:minimizer_lower_bound}.
  By the product and quotient rules, $\mathrm{Rhs}_{\delta_t}$ is a sum of terms each containing two factors among $\partial_x\left(\diff_{t + \delta_t} - \diff_t\right)$ and $\partial_x^2\left(\diff_{t + \delta_t} - \diff_t\right)$, and is therefore quadratic in $\diff_{t + \delta_t} - \diff_t$.
  Thus, for $\delta_t$ small enough, we can estimate
  \begin{equation}
    \|\mathrm{Rhs}_{\delta_t}\|_{L^\infty} \leq C(a, b, \alpha, \norm{t u}_{L^{\infty}}, R, \rmeas) \norm{\diff_{t + \delta_t} - \diff_{t}}_{W^{2, \infty}}^2,
  \end{equation}
  for some constant $C$ depending on $a, b, \alpha, R, t, u, \rmeas$. 
  By \cref{lem:sobolev_stability_of_minimizers}, this implies that 
  \begin{equation}
    \|\mathrm{Rhs}_{\delta_t}\|_{L^\infty} \leq \delta_t^2 C(a, b, \alpha, R, \rmeas) \norm{\uvel}_{L^\infty}^2,
  \end{equation}
  from which the result follows by \cref{lem:elliptic_pde_solution}.
  For $k > 0$, applying \cref{lem:sobolev_stability_of_minimizers} to $x$-derivatives of $\mathrm{Rhs}$  yields
  \begin{equation}
    \|\mathrm{Rhs}_{\delta_t}\|_{W^{k, \infty}} \leq \delta_t^2 C(a, b, \alpha, R, \rmeas) \norm{\uvel}_{W^{k, \infty}}^2,
  \end{equation}
  allowing us to again conclude by \cref{lem:elliptic_pde_solution}.
  The proof of the result for $C^k, C^{k + 1}, C^{k + 2}$ is analogous.
\end{proof}
\subsection{PDE solution}
\label{sec:first_order_pde_solution}
The time derivative derived in \cref{thm:time_derivative} allows us to show that the one-parameter family of minimizers $t \mapsto \diff_t$ solves the PDE in \cref{eqn:lagrangian_eqn}.
As in the previous section, we assume that the probability measure $\rmeas$ on $[a,b]$ is absolutely continuous with respect to $\Leb$ and denote as $\rmeas(\cdot)$ its Lebesgue density satisfying $0 < \essinf_{[a, b]} \rmeas$, $\esssup_{[a, b]} \rmeas < \infty$, and $\rmeas\in W^{1,\infty}$.
\begin{lemma}[Continuity with respect to parameters]
  \label[lemma]{lem:continuity_parameters}
  For $k \geq 0$, let $\rmeas \in W^{k + 1, \infty}$ and $\ugen \in W^{k, \infty}$.
  For $A\in W^{k + 1, \infty}$ with $0 < C_{\min} \leq A \leq C_{\max}$, let $\theta_{A}$ be the solution of the divergence form elliptic PDE
  \begin{equation}
    \theta_{A} \rmeas - \alpha \partial_x \left(A \partial_x \theta_{A}  \rmeas \right) = \ugen \rmeas.
  \end{equation}
  The map $A \mapsto \theta_{A}$ is locally Lipschitz continuous as a map from $W^{k + 1, \infty}$ to $W^{k + 2, \infty}$.
  Analogous result holds for $W^{k, \infty}, W^{k + 1, \infty}, W^{k + 2,\infty}$ replaced with $C^k, C^{k + 1}, C^{k + 2}$.
\end{lemma}
\begin{proof}
  By \cref{lem:elliptic_pde_solution}, the map $A \mapsto \theta_{A}$ is well-defined. 
  To prove its Lipschitz continuity, we consider $A_1, A_2 \in W^{1, \infty}$ in a neighborhood of $A$ and $\theta_1 = \theta_{A_1}, \theta_2 = \theta_{A_2}$.
  Subtracting the two equations, we obtain
  \begin{equation}
    \label{eqn:theta_equation_cont}
    \left(\theta_{1} - \theta_{2}\right)\rmeas - \alpha \partial_x \left(A_1 \partial_x \left(\theta_{1} - \theta_{2} \right) \rmeas\right) 
    = - \alpha \partial_x \left( \left(A_2 - A_1\right) (\partial_x \theta_{2}) \rmeas\right).
  \end{equation}
  The $L^\infty$ norm of the right-hand side is bounded by $\norm{A_2 - A_1}_{W^{1, \infty}}$ times a constant that depends on $\alpha$ and $\rmeas$, from which the result follows. 
  Here, we have used \cref{lem:elliptic_pde_solution} to upper-bound (derivatives of) $\theta_2$.
  The result for $k>0$ follows by taking derivatives of the right-hand side of \cref{eqn:theta_equation_cont} to bound its $W^{k, \infty}$ norm by $\norm{A_2 - A_1}_{W^{k + 1, \infty}}$ and applying \cref{lem:elliptic_pde_solution}.
  The proof for $C^k, C^{k + 1}, C^{k + 2}$ is analogous.
\end{proof}
\begin{theorem}
  \label{thm:first_order_PDE}
  For $k \geq 0$, let $\rmeas \in W^{k + 1, \infty}$ and $\uvel \in W^{k + 2, \infty}$.
  For  
  \begin{equation}
    f_{\alpha}^{u_0}(\diff) \coloneqq \int \limits_{a}^{b} \diff \cdot u_0 + \alpha \partial_x \diff \cdot \left(\partial_x u_0 \right) \D \rmeas(x) 
    - \int \limits_{a}^{b} \alpha \partial_x \diff \D \rmeas(x)
  \end{equation}
  and every $t > 0$, let $\diff_t$ be the minimizer of $F_{\alpha}^{f_{\alpha}^{t \uvel}, \rmeas}(\cdot)$ in $\Helly$.
  Then, $t \mapsto \diff_{t}$ is in $C^{1}\left([0, \infty); W^{k + 2, \infty}\right)$ and solves the first order Lagrangian PDE 
\begin{equation}
    \label{eqn:first_order_lagrangian_eqn}
    \begin{cases}
        \dot{\diff}_t \rmeas - \alpha \partial_x \left(\frac{\partial_x \dot{\diff}_t}{\left(\partial_x \diff_t\right)^2} \rmeas\right) - \uvel \rmeas +  \alpha \partial_x \left((\partial_x \uvel) \rmeas\right)= 0 \\
        \diff_0 = \mathrm{Id},
        \quad \quad \forall t \geq 0: \diff_{t}(a) = a, \quad \diff_{t}(b) = b.
    \end{cases}
\end{equation}
  The analogous result holds with $W^{k + 1, \infty}, W^{k + 2, \infty}$ replaced with $C^{k + 1}, C^{k + 2}$.
\end{theorem}
\begin{proof}
  The assumptions on $\rmeas$ and $\uvel$ ensure that $f_{\alpha}^{t \uvel}$ is represented by integration against a $W^{k, \infty}$ function of the form $w + t \uvel_{\alpha}$ as required by \cref{thm:time_derivative}. 
  The term $- \alpha \int \partial_x \diff \D \rmeas$ contributes the $t$-independent part $w = \alpha \partial_x \log \rmeas$, and the velocity its $t$-scaled part $\uvel_{\alpha}$. Both lie in $W^{k, \infty}$ since $\rmeas \in W^{k + 1, \infty}$ with $\essinf \rmeas > 0$ and $\uvel \in W^{k + 2, \infty}$.
  By \cref{sec:existence_minimal_regularity}, the minimizers $\diff_t$ exist and by \cref{lem:higher_regularity} they are in $W^{2, \infty}$.
  By \cref{thm:time_derivative} the derivatives $\dot{\diff}_t$ satisfy the PDE. 
  The upper bound of $\dot{\diff}_t$ in $W^{k + 2, \infty}$ follows from \cref{lem:elliptic_pde_solution}.
  The continuity of $t \mapsto \diff_t$ follows from \cref{lem:continuity_parameters}.
  The proof for $C^{k + 1}, C^{k + 2}$ is analogous.
\end{proof}
Eulerian solutions require second-order time derivatives. 
Proving their existence requires differentiability of solutions of elliptic PDEs with respect to their coefficients.
This is a standard result but usually not presented for $W^{2, \infty}$ or $C^2$ regularity since these spaces are not suitable for multidimensional elliptic regularity theory. 
\begin{lemma}[Differentiability with respect to parameters]
  \label[lemma]{lem:differentiability_parameters}
  For $k \geq 0$, let $\ugen \in W^{k, \infty}$.
  For $A\in W^{k + 1, \infty}$ with $0 < C_{\min} \leq A \leq C_{\max}$ and $\tfrac{\D \rmeas}{\D \Leb} \in W^{k + 1, \infty},$ let $\theta_{A}$ be the solution of the divergence form elliptic PDE
  \begin{equation}
    \theta_{A} \rmeas - \alpha \partial_x \left(A \partial_x \theta_{A} \rmeas \right) = \ugen \rmeas.
  \end{equation}
  For $B$ in a neighborhood of $A$, define $\bar{\theta}$ as the solution of the elliptic PDE
  \begin{equation}
    \label{eqn:theta_bar_equation}
    \bar{\theta} \rmeas - \alpha \partial_x \left(A \partial_x \bar{\theta}  \rmeas \right) = \alpha \partial_x \left((B - A) \partial_x \theta_A \rmeas \right).
  \end{equation}
  Then, we have 
  \begin{equation}
    \lim \limits_{\norm{B - A}_{W^{k + 1, \infty}} \rightarrow 0} \frac{\norm{\theta_B - \theta_A - \bar{\theta}}_{W^{k + 2, \infty}}}{\norm{B - A}_{W^{k + 1, \infty}}} = 0.
  \end{equation}
  Thus, $A \mapsto \theta_A$ is Fr\'echet differentiable from $W^{k + 1, \infty}$ to $W^{k + 2, \infty}$ in $A$ with derivative given by $(B - A) \mapsto \bar{\theta}$.
  The analogous result holds with $W^{k, \infty}, W^{k + 1, \infty}, W^{k + 2,\infty}$ replaced with $C^k, C^{k + 1}, C^{k + 2}$.
\end{lemma}
\begin{proof}
  Subtracting the equations defining $\theta_B$ and $\theta_A$, we obtain 
  \begin{equation}
    \left(\theta_B - \theta_A\right)\frac{\D \rmeas}{\D \Leb} - \alpha \partial_x \left(A \partial_x \left(\theta_B - \theta_A \right) \frac{\D \rmeas}{\D \Leb}\right) 
    = \alpha \partial_x \left( \left(B - A\right) \partial_x \theta_B \frac{\D \rmeas}{\D \Leb}\right).
  \end{equation}
  Subtracting the equation defining $\bar{\theta}$, we obtain
  \begin{equation}
    \left(\theta_B - \theta_A - \bar{\theta}\right)\frac{\D \rmeas}{\D \Leb} - \alpha \partial_x \left(A \partial_x \left(\theta_B - \theta_A - \bar{\theta} \right) \frac{\D \rmeas}{\D \Leb}\right) 
    = \alpha \partial_x \left( \left(B - A\right) \partial_x \left(\theta_B  - \theta_A \right) \frac{\D \rmeas}{\D \Leb}\right).
  \end{equation}
  The right hand side is in $W^{k, \infty}$ with a norm upper bounded by $\norm{B - A}_{W^{k + 1, \infty}}\norm{\theta_B - \theta_A}_{W^{k + 2, \infty}}$ times a constant.
  Thus, the result follows from \cref{lem:elliptic_pde_solution,lem:continuity_parameters}.
  The proof for $C^k, C^{k + 1}, C^{k + 2}$ is analogous.
\end{proof}
We now have all the differentiability properties to prove the existence of solutions to the Lagrangian and Eulerian pressureless IGR equations.
\begin{theorem}
  \label{thm:second_order_PDE}
  For $k \geq 0$, let the probability measure $\rmeas$ have a density in $W^{k + 1, \infty}$ with respect to the Lebesgue measure and let $\uvel_0 \in W^{k + 2, \infty}$.
  For  
  \begin{equation}
    f_{\alpha}^{u_0}(\diff) \coloneqq \int \limits_{a}^{b} \diff \cdot u_0 + \alpha \partial_x \diff \cdot \left(\partial_x u_0 \right) \D \rmeas(x) 
    - \int \limits_{a}^{b} \alpha \partial_x \diff \D \rmeas(x)
  \end{equation}
  and every $t > 0$, let $\diff_t$ be the minimizer of $F_{\alpha}^{f_{\alpha}^{t \uvel_0}, \rmeas}(\cdot)$ in $\Helly$.
  Then, $t \mapsto \diff_{t}$ is in $C^{2}\left([0, \infty); W^{k + 2, \infty}\right)$ and solves the Lagrangian pressureless IGR equation \cref{eqn:lagrangian_eqn}, 
\begin{equation}
    \begin{cases}
        \ddot{\diff}_t \frac{\D \rmeas}{\D \Leb} - \alpha \partial_x\left(\left[\partial_x \diff_t \right]^{-2} [\partial_x (\ddot{\diff}_t)] \frac{\D \rmeas}{\D \Leb}\right) 
        + 2 \alpha \partial_x \left(\left[\partial_x \diff_t \right]^{-3}  [\partial_x \dot{\diff}_t]^2 \frac{\D \rmeas}{\D \Leb}\right) = 0,\\
        \dot{\diff}_0 = \uvel_0 
        \quad \quad \diff_0 = \mathrm{Id},
        \quad \quad \forall t \geq 0: \diff_{t}(a) = a, \quad \diff_{t}(b) = b.
    \end{cases}
\end{equation}
  The analog result holds true for $W^{k, \infty}, W^{k + 1, \infty}, W^{k + 2,\infty}$ replaced with $C^k, C^{k + 1}, C^{k + 2}$.
\end{theorem}
\begin{proof}
  Using \cref{lem:differentiability_parameters}, we can obtain $\ddot{\diff}$ as the derivative of the solution of \cref{eqn:first_order_lagrangian_eqn} with respect to the solution field $\rmeas / (\partial_x \diff_t)^2$.
  The equation~\cref{eqn:theta_bar_equation} then yields the Lagrangian pressureless IGR equation.
\end{proof}
We can translate the above to Eulerian coordinates, obtaining the following theorem.
\begin{theorem}
  \label{thm:eulerian_PDE}
  For $k \geq 0$, let  $\rmeas \in C^{k + 1}([a, b])$ and $\uvel_0 \in C^{k + 2}([a, b])$.
  For  
  \begin{equation}
    f_{\alpha}^{u_0}(\diff) \coloneqq \int \limits_{a}^{b} \diff \cdot u_0 + \alpha \partial_x \diff \cdot \left(\partial_x u_0 \right) \D \rmeas(x) 
    - \int \limits_{a}^{b} \alpha \partial_x \diff \D \rmeas(x)
  \end{equation}
  and every $t > 0$, let $\diff_t$ minimize $F_{\alpha}^{f_{\alpha}^{t \uvel_0}, \rmeas}(\cdot)$ in $\Helly$.
  Define the Eulerian quantities $\uvel, \rho$ by $\uvel(x, t) \coloneqq \dot{\diff}_t \left(\diff_{t}^{-1}(x)\right)$ and $\rho(x, t) \coloneqq \left(\frac{\rmeas}{\partial_x \diff_t} \right)\left(\diff_{t}^{-1}(x)\right)$.
  Then, $t \mapsto \uvel(\cdot, t)$ is in the space $C^{0}\left([0, \infty); C^{k + 2}([a, b])\right) \cap C^{1}\left([0, \infty); C^{k + 1}([a, b])\right)$ and $t \mapsto \rho(\cdot, t)$ is in $C^{0}\left([0, \infty); C^{k + 1}([a, b])\right) \cap C^{1}\left([0, \infty); C^{k}([a, b])\right)$. 
  Together, they solve the Eulerian pressureless IGR equation \cref{eqn:1d-igr-pressureless},
\begin{equation}
    \begin{cases}
         \partial_t
         \begin{pmatrix}
              \rho \uvel \\
              \rho
         \end{pmatrix}
         + \partial_x
         \begin{pmatrix}
              \rho \uvel^2 + \Sigma \\ 
              \rho \uvel
         \end{pmatrix}
         = 
         \begin{pmatrix} 
              0\\
              0
         \end{pmatrix}, \quad &\textup{for} \quad x \in (a, b), t \in \R_+,\\
         \rho^{-1} \Sigma {- \alpha \partial_x(\rho^{-1} \partial_x \Sigma)} = 2 \alpha \left(\partial_x \uvel\right)^2, \quad &\textup{for} \quad x \in (a, b), t \in \R_+,\\
         \uvel(a, t) = \uvel(b, t) = \partial_x \Sigma(a,t) = \partial_x \Sigma(b, t) = 0, \quad &\text{for} \quad t \in \R_+, \\
         \uvel(x, 0) = \uvel_0(x), \quad \rho(x, 0) = \frac{\D \rmeas}{\D \Leb}(x), \quad &\textup{for} \quad x \in [a,b].
    \end{cases}
\end{equation}
\end{theorem}
\begin{lemma}
  \label[lemma]{lem:inverse_chain_continuous}
  For $k \geq 0$, let $t \mapsto g_t$ be in $C^{1}\left([0, \infty); C^{k + 1}([a, b])\right)$ and $t \mapsto h_t$ be in $C^{1}\left([0, \infty); C^{k + 1}([a, b])\right)$.
  Furthermore, assume that for all $t$, $0 < c \leq \partial_x h_t \leq C < \infty$ and $h_t$ is a monotone bijection from $[a ,b]$ onto itself.
  Then, the following holds.
  \begin{enumerate}[label=(\roman*)]
    
    \item The maps $t \mapsto g_t \circ h_t$ and $t \mapsto g_t \circ h_t^{-1}$ are in $C^{0}\left([0, \infty); C^{k + 1}([a, b])\right)$ with spatial derivatives
    \begin{equation}
      \label{eqn:xderivative_concatenation}
      \partial_x \left(g_t \circ h_t\right) 
      = \partial_x g_t \circ h_t \cdot \partial_x h_t 
      \quad 
      \text{and}
      \quad 
      \partial_x \left(g_t \circ h_t^{-1}\right) 
      = \frac{\partial_x g_t }{\partial_x h_t} \circ h_t^{-1} 
    \end{equation}
    \item The map $t \mapsto g_t \circ h_t^{-1}$ is in  $C^{1}\left([0, \infty); C^{k}([a, b])\right)$ with time derivatives given by
    \begin{equation}
      \partial_t \left(g_t \circ h_t^{-1}\right) 
      = \left(\partial_t g_t - \partial_x g_t \frac{\partial_t h_t}{\partial_x h_t}\right) \circ h_t^{-1}.
    \end{equation} 
  \end{enumerate}
\end{lemma}
\begin{proof}
  We prove the result for $k = 0$, the proof for $k > 0$ is analogous.
 
  (i): 
  The inverse function theorem implies that $h_t^{-1}$ is in $C^{1}([a, b])$ with derivative given by $1 /(\partial_x h_t) \circ h_t^{-1}$. 
  Thus, the chain rule implies that $g_t \circ h_t^{-1}$ is in $C^1([a, b])$ with derivative given by \cref{eqn:xderivative_concatenation}. 
  Since $g_t$ and $\frac{\partial_x g_t}{\partial_x h_t}$ are continuous over the compact set $[a, b]$, they are in particular uniformly continuous and thus $g_t \circ h_{t + s}^{-1}$ converges to $g_t \circ h_t^{-1}$ in $C^1$, for $s \to 0$.  
  Thus, by continuity of $t \mapsto g_t$ in $C^1$, 
  \begin{equation} 
    \begin{split}
    & \|g_{t + s} \circ h_{t + s}^{-1} -  g_{t} \circ h_{t}^{-1} \|_{C^1} 
    = \|g_{t + s} \circ h_{t + s}^{-1} - g_{t} \circ h_{t + s}^{-1} + g_{t} \circ h_{t + s}^{-1} -  g_{t} \circ h_{t}^{-1} \|_{C^1} \\
    \leq & \left\|g_{t + s} \circ h_{t + s}^{-1} - g_{t} \circ h_{t + s}^{-1}\right\|_{C^1} + \left\|g_{t} \circ h_{t + s}^{-1} -  g_{t} \circ h_{t}^{-1} \right\|_{C^1} \xrightarrow[]{s \to 0} 0.
    \end{split}
  \end{equation}
  (ii):
  Membership of $t \mapsto h_t$ in $C^1([0, \infty); C^{0 + 1})$ implies differentiability of $(x, t) \mapsto h_t(x)$. 
  Thus, the implicit function theorem implies that $\partial_t\left(h_t^{-1}(x)\right) = - \frac{\partial_t h_t}{\partial_x h_t} \circ h_t^{-1}$. 
  By the chain rule, we thus have 
    \begin{equation}
      \label{eqn:temp_derivative_concatenation}
    \partial_s\big|_{s = 0} \left(g_t \circ h_{t + s}^{-1}\right) =  - \partial_x g_t \frac{\partial_t h_t}{\partial_x h_t} \circ h_t^{-1}.
    \end{equation}
    To show that $g_t \circ h_t^{-1}$ is indeed in $C^1([0, \infty); C^0([a, b]))$, an empty sum and the triangle inequality yield
    \begin{equation} 
      \begin{split}
      & \frac{\left\|g_{t + s} \circ h_{t + s}^{-1} -  g_{t} \circ h_{t}^{-1}  - s \left(\partial_t g_t - \partial_x g_t \frac{\partial_t h_t}{\partial_x h_t}\right) \circ h_t^{-1}\right\|_{C^0}}{s} 
      \leq \frac{\left\|g_{t + s} \circ h_{t + s}^{-1} - g_{t} \circ h_{t + s}^{-1} - s \partial_t g_t \circ h_{t + s}^{-1}\right\|_{C^0}}{s} \\
      &+ \frac{\left\| g_{t} \circ h_{t + s}^{-1} -  g_{t} \circ h_{t}^{-1}  + s \partial_x g_t \frac{\partial_t h_t}{\partial_x h_t} \circ h_t^{-1} \right\|_{C^0}}{s} 
      + \frac{\left\|s(\partial_t g_t \circ h_{t + s}^{-1} - \partial_t g_t \circ h_{t}^{-1})\right\|_{C^0}}{s}.
      \end{split}
    \end{equation}
    The first summand goes to zero by temporal regularity of $g$ and the last summand by spatial regularity of $\partial_t g$.
    For the second term, the mean value theorem and \cref{eqn:temp_derivative_concatenation} yield, for a $\hat{s} \in [0, s]$ depending on $x$ and $s$,
    \begin{equation}
      \frac{\left\| g_{t} \circ h_{t + s}^{-1} -  g_{t} \circ h_{t}^{-1}  + s \partial_x g_t \frac{\partial_t h_t}{\partial_x h_t} \circ h_t^{-1} \right\|_{C^0}}{s} 
      = 
      \frac{\left\| - s \partial_x g_{t} \frac{\partial_t h_{t + \hat{s}}}{\partial_x h_{t + \hat{s}}} \circ h_{t + \hat{s}}^{-1} + s \partial_x g_t \frac{\partial_t h_t}{\partial_x h_t} \circ h_t^{-1} \right\|_{C^0}}{s}. 
    \end{equation}
    For every $T > 0$, $(t, x) \mapsto \partial_x g_t(x) \frac{\partial_t h_t(x)}{\partial_x h_t(x)} \circ h_{t + \hat{s}}^{-1}$ is continuous on the compact set $[0, T) \times [a, b]$ and thus uniformly continuous. 
    Thus, the right hand side converges to zero as $s \to 0$.
\end{proof}
\begin{proof}[Proof of \cref{thm:eulerian_PDE}]
  By \cref{thm:second_order_PDE}, $t \mapsto \diff_{t}$ is in $C^{2}\left([0, \infty); C^{k + 2}([a, b])\right)$ and solves equation \cref{eqn:lagrangian_eqn},
  \begin{equation}
      \begin{cases}
          \ddot{\diff}_t \frac{\D \rmeas}{\D \Leb} - \alpha \partial_x\left(\left[\partial_x \diff_t \right]^{-2} [\partial_x (\ddot{\diff}_t)] \frac{\D \rmeas}{\D \Leb}\right) 
          + 2 \alpha \partial_x \left(\left[\partial_x \diff_t \right]^{-3}  [\partial_x \dot{\diff}_t]^2 \frac{\D \rmeas}{\D \Leb}\right) = 0,\\
          \dot{\diff}_0 = \uvel_0 
          \quad \quad \diff_0 = \mathrm{Id},
          \quad \quad \forall t \geq 0: \diff_{t}(a) = a, \quad \diff_{t}(b) = b.
      \end{cases}
  \end{equation}
  As a minimizer of $F_{\alpha}^{f_{\alpha}^{t \uvel_0}, \rmeas}(\cdot)$, $\diff_t$ is strictly increasing and thus invertible.
  By \cref{le:minimizer_lower_bound}, $\partial_x \diff_t$ is uniformly bounded from below away from zero on $[a, b]$ and any compact set of $t$-values.
  Likewise, by \cref{le:rho_lower_bound}, $\partial_x \diff_t$ is uniformly bounded from above on $[a, b]$ and any compact set of $t$-values.
  Thus, \cref{lem:inverse_chain_continuous} implies that 
  $t \mapsto \uvel(\cdot, t)$ is in $C^{0}\left([0, \infty); C^{k + 2}([a, b])\right) \cap C^{1}\left([0, \infty); C^{k + 1}([a, b])\right)$ and that
  $t \mapsto \rho(\cdot, t)$ is in $C^{0}\left([0, \infty); C^{k + 1}([a, b])\right) \cap C^{1}\left([0, \infty); C^{k}([a, b])\right)$.
  By \cref{lem:inverse_chain_continuous}, their spatial derivatives are given by 
  \begin{equation}
    \partial_x \uvel = \frac{\partial_x \dot{\diff}_t}{\partial_x \diff_t} \circ \diff_{t}^{-1}(x)
    \quad \text{and} \quad
    \partial_x \rho = \left(\frac{\partial_x\left(\rmeas  / \partial_x \diff_t\right)}{\left(\partial_x \diff_t\right)}\right) \circ \diff_{t}^{-1}(x).
  \end{equation}
  We begin by verifying the mass transport equation $\partial_t \rho + \partial_x(\rho \uvel) = 0$.
  By \cref{lem:inverse_chain_continuous}, we have
  \begin{equation}
    \begin{split}
      \partial_t \rho 
      =& \partial_t \left(\frac{\rmeas}{\partial_x \diff_t} \circ \diff_{t}^{-1}\right)
      = \left(\partial_t \left(\frac{\rmeas}{\partial_x \diff_t}\right) 
        - \partial_x \left(\frac{\rmeas}{\partial_x \diff_t}\right) \frac{\dot{\diff}_t}{\partial_x \diff_t} \right) \circ \diff_{t}^{-1}\\
      =& \left(- \left(\frac{\rmeas \partial_x \dot{\diff}_t}{\left(\partial_x \diff_t\right)^2}\right) 
        - \partial_x \left(\frac{\rmeas}{\partial_x \diff_t}\right) \frac{\dot{\diff}_t}{\partial_x \diff_t} \right) \circ \diff_{t}^{-1}
      = - \rho \partial_x \uvel - \partial_x (\rho) \uvel,
    \end{split}
  \end{equation}
  verifying the mass transport equation.
  From the spatial derivative of $\uvel$, we obtain
  \begin{equation}
    \partial_x \uvel = \frac{\partial_x \dot{\diff}_t}{\partial_x \diff_t} \circ \diff_{t}^{-1}(x)
    \Rightarrow 
    \partial_x \left(\left[\partial_x \diff_t \right]^{-3}  [\partial_x \dot{\diff}_t]^2 \frac{\D \rmeas}{\D \Leb}\right) 
    =  \partial_x \left(\rho(\partial_x \uvel)^2  \circ \diff_t \right)
  \end{equation}
  Using \cref{lem:inverse_chain_continuous} again to compute the temporal derivative of $\uvel$, we obtain
  \begin{equation}
    \partial_t \uvel 
    = \partial_t \left(\dot{\diff}_t \circ \diff_t^{-1}\right)
    = \left(\partial_t \dot{\diff}_t - \partial_x \dot{\diff}_t \frac{\dot{\diff}_t}{\partial_x \diff_t}\right) \circ \diff_{t}^{-1}(x)
    \Rightarrow 
    \ddot{\diff}_t  = \left(\partial_t \uvel + \partial_x \uvel \cdot \uvel\right) \circ \diff_t.
  \end{equation}
  Combining these identities with the Lagrangian equation \cref{eqn:lagrangian_eqn}, we obtain
  \begin{equation}
    \left(\partial_t \uvel + \partial_x \uvel \cdot \uvel\right) \circ \diff_t \frac{\D \rmeas}{\D \Leb}
    - \alpha \partial_x \left(\rho \partial_x \left(\partial_t \uvel + \partial_x \uvel \cdot \uvel\right) \circ \diff_t \right)
    + 2 \alpha \partial_x \left(\rho (\partial_x \uvel)^2  \circ \diff_t \right) = 0.
  \end{equation} 
  Using again \cref{lem:inverse_chain_continuous}, we obtain 
  \begin{equation}
    \partial_x \left( \rho \partial_x \left(\partial_t \uvel + \partial_x \uvel \cdot \uvel\right) \circ \diff_t \right)
    = 
    \partial_x \left( \rho \partial_x \left(\partial_t \uvel + \partial_x \uvel \cdot \uvel\right) \right) \circ \diff_t \cdot \partial_x \diff_t
  \end{equation}
  and
  \begin{equation}
  \partial_x \left(\rho (\partial_x \uvel)^2  \circ \diff_t \right)
    = 
    \partial_x \left(\rho (\partial_x \uvel)^2 \right)\circ \diff_t  \cdot \partial_x \diff_t, 
  \end{equation}
  resulting in 
  \begin{equation}
    \left(\partial_t \uvel + \partial_x \uvel \cdot \uvel\right) \rho
    - \alpha \partial_x \left(\rho \partial_x \left(\partial_t \uvel + \partial_x \uvel \cdot \uvel\right) \right) 
    + 2 \alpha \partial_x \left(\rho (\partial_x \uvel)^2 \right) = 0.
  \end{equation} 
  By the mass transport equation, we have
  \begin{equation}
    \partial_t \rho + \partial_x (\rho \uvel) = 0 \Rightarrow \left(\partial_t \uvel + \partial_x \uvel \cdot \uvel\right) \rho = \partial_t (\rho \uvel) + \partial_x (\rho \uvel^2),
  \end{equation}
  resulting in 
  \begin{equation}
    \left(\partial_t(\rho \uvel) + \partial_x (\rho \uvel^2)\right)
    - \alpha \partial_x \left(\rho \partial_x \left(\frac{\partial_t(\rho \uvel) + \partial_x (\rho \uvel^2)}{\rho}\right) \right) 
    + 2 \alpha \partial_x \left(\rho (\partial_x \uvel)^2 \right) = 0.
  \end{equation} 
  Defining $\Sigma$ as
  \begin{equation}
    \Sigma(\cdot) 
    \coloneqq 
    - \alpha \left[\rho \partial_x \left(\frac{\partial_t(\rho \uvel) + \partial_x (\rho \uvel^2)}{\rho}\right)\right](a)
    + 2 \alpha \left[\rho (\partial_x \uvel)^2\right](a)
    -   \int \limits_{a}^{(\cdot)} \partial_t (\rho \uvel) + \partial_x (\rho \uvel^2) \D x, 
  \end{equation}
  we observe that $\partial_x \Sigma$ is zero in $x = a, b$ since $\uvel$ vanishes at the boundary.
  We now obtain 
  \begin{equation}
     \left(\partial_x \Sigma\right) 
    - \alpha \partial_x \left(\rho \partial_x \left(\frac{\partial_x \Sigma}{\rho}\right) \right) 
    - 2 \alpha \partial_x \left(\rho (\partial_x \uvel)^2 \right) = 0
    \implies
     \partial_x \left(\Sigma 
    - \alpha \left(\rho \partial_x \left(\frac{\partial_x \Sigma}{\rho}\right) \right) 
    - 2 \alpha \left(\rho (\partial_x \uvel)^2 \right)\right) = 0.
  \end{equation}
The constant offset part of $\Sigma$ was chosen such that the function vanishes at $x = a$.
Thus we have
\begin{equation}
  \Sigma 
  - \alpha \left(\rho \partial_x \left(\frac{\partial_x \Sigma}{\rho}\right) \right) 
  - 2 \alpha \left(\rho (\partial_x \uvel)^2 \right) = 0 
  \implies
  \frac{\Sigma}{\rho}
  - \alpha \left(\partial_x \left(\frac{\partial_x \Sigma}{\rho}\right) \right) 
  - 2 \alpha \left(\partial_x \uvel\right)^2 = 0 
\end{equation}
By uniqueness of solutions to this second order elliptic equation with Neumann data, we conclude that indeed  
\begin{equation}
  \partial_t (\rho \uvel) + \partial_x (\rho \uvel^2 + \Sigma) = 0,
\end{equation}
for $\Sigma$ defined as 
\begin{equation}
  \rho^{-1} \Sigma {- \alpha \partial_x(\rho^{-1} \partial_x \Sigma)} = 2 \alpha \left(\partial_x \uvel\right)^2, \quad \text{with} \quad \partial_x \Sigma(a,t) = \partial_x \Sigma(b, t) = 0.
\end{equation}
\end{proof}
\begin{remark}[Geodesic completeness of information geometrically regularized diffeomorphism manifolds]
  \cref{thm:second_order_PDE,thm:eulerian_PDE} imply that for $k \geq 2$, manifolds of unidimensional $k$-times differentiable diffeomorphisms are geodesically complete with respect to the dual geodesics induced by the potential 
  \begin{equation}
  \bpot(\diff) \coloneqq \int \limits_{a}^{b} \frac{\left(\diff_t(x) - x\right)^2}{2} - \alpha \log\left(\partial_x \diff\right) \D \rmeas(x).
  \end{equation}
  To this end, it is enough to observe that the geodesic starting in $\bar{\diff}$ in direction $\uvel$ is given by $t \mapsto \diff_t \circ \bar{\diff}$, where $\diff_t$ satisfies \cref{eqn:lagrangian_eqn} with $\uvel_0 = \uvel \circ \bar{\diff}$ and $\rmeas = \bar{\diff}_{\#} \Leb$.
\end{remark}

\section{Summary, conclusion, and outlook}
\label{sec:conclusion}
This work is a first step toward a rigorous understanding of information geometric regularization. 
We focus on the unidimensional pressureless case, and establish the existence of solutions in $C^{k}$ that match the regularity of the initial velocity field. 
By means of $\Gamma$-convergence, we also prove that IGR solutions converge to entropy solutions of the pressureless Euler equation in the limit $\alpha \to 0$.
Important future directions remain, most notably the inclusion of pressure forces and study of multidimensional problems.
In these cases, IGR lacks a straightforward variational form such as the one used in \cref{eqn:variational_exp}. 
We anticipate that these extensions will require techniques more akin to those ordinarily used in differential equations and evolutionary PDEs, such as fixed point arguments and energy estimates.
Even for the unidimensional pressureless case considered in this work, these techniques seem more natural for proving uniqueness of solutions, their continuous dependence on the initial data, and quantitative convergence results for $\alpha \to 0$ in the absence of shock waves.
We therefore defer these exciting directions to future work.

\section*{Acknowledgments}
The authors gratefully acknowledge support from the Air Force Office of Scientific Research under award number FA9550-23-1-0668 (Information Geometric Regularization for Simulation and Optimization of Supersonic Flow). RC acknowledges support through a PURA travel award from the Georgia Tech Office of Undergraduate Education. 
We thank the anonymous referees for their valuable suggestions that have helped us improve the article.
We also thank Dian Guan for his careful reading of this manuscript. 
This work was produced using Microsoft Copilot (autocompletion, spelling/grammar) and Claude Code and Codex (additional proofreading).
\bibliography{references}
\bibliographystyle{plain}

\appendix

\end{document}